\documentclass{sinst}

\usepackage{amsbsy,amsfonts,amsmath,amssymb,amsthm}
\usepackage{array}
\usepackage{bm}
\usepackage{color}
\usepackage{enumerate}
\usepackage{mathrsfs}
\usepackage{latexsym}
\usepackage[colorlinks=true]{hyperref}
\usepackage{mathtools}
\mathtoolsset{showonlyrefs, showmanualtags}
\hypersetup{urlcolor=blue, citecolor=blue, linkcolor=blue}

\definecolor{wineRed}{rgb}{0.7,0,0.3}
\definecolor{grandBleu}{rgb}{0,0,0.8}
\definecolor{darkGreen}{rgb}{0,0.4,0}
\definecolor{blueViolet}{rgb}{0.4,0,1.0}
\definecolor{bloodOrange}{rgb}{0.85,0.05,0}
\definecolor{mycolor}{rgb}{0.8,0,0.2}

\usepackage[textsize=small]{todonotes}
\usepackage{cite}

\usepackage{comment}

\DeclareMathAlphabet{\mathpzc}{OT1}{pzc}{m}{it}

\usepackage[textsize=small]{todonotes}
\numberwithin{equation}{section}

\theoremstyle{plain}
\newtheorem{mTh}{Main Theorem} 

\newtheorem{keylem}{Key-Lemma}
\newtheorem{cor}{Corollary}

\theoremstyle{definition}
\newtheorem{defn}{Definition}
\newtheorem{rem}{Remark}

\newtheorem{ex}{Example}

\def\N{\mathbb{N}}
\def\R{\mathbb{R}}

\def\sH{\mathscr{H}}
\def\sJ{\mathscr{J}}
\def\sK{\mathscr{K}}
\def\sL{\mathscr{L}}
\def\sP{\mathscr{P}}
\def\sQ{\mathscr{Q}}

\def\sS{\mathscr{S}}
\def\sT{\mathscr{T}}
\def\sU{\mathscr{U}}
\def\sV{\mathscr{V}}
\def\sW{\mathscr{W}}
\def\sX{\mathscr{X}}
\def\sZ{\mathscr{Z}}
\def\ad{{\rm ad}}
\def\ds{\displaystyle}

\def\dual#1{\langle #1 \rangle}
\def\Lap{\mathit{\Delta}}
\def\Sgn{\mathop{\mathrm{Sgr}}\nolimits}

\DeclareMathOperator{\diver}{div}
\DeclareMathOperator{\proj}{proj}

\begin{document}
\vspace*{-0cm}
\title{Optimal Control of Pseudo-Parabolic KWC Systems for Grain Boundary Motion
\thanks{This work is supported by JST SPRING Grant Number JPMJSP2109. This work is partially supported by the Office of Naval Research (ONR) under Award NO: N00014-24-1-2147,
NSF grant DMS-2408877, the Air Force Office of Scientific Research (AFOSR) under Award NO: FA9550-22-1-0248.}
\vspace{-3ex}
}
\author{Harbir Antil
\footnotemark[2]
}
\affiliation{Department of Mathematical Sciences and \\ The Center for Mathematics and Artificial Intelligence (CMAI), 
\\
George Mason University, Fairfax, VA 22030, USA}
\email{hantil@gmu.edu}
\vspace{-2ex}

\sauthor{Daiki Mizuno
\footnotemark[2]
}
\saffiliation{Division of Mathematics and Informatics, \\ Graduate School of Science and Engineering, Chiba University, \\ 1-33, Yayoi-cho, Inage-ku, 263-8522, Chiba, Japan}
\semail{d-mizuno@chiba-u.jp}
\vspace{-2ex}

\tauthor{Ken Shirakawa
\footnotemark[3]
}
\taffiliation{Department of Mathematics, Faculty of Education, Chiba University \\ 1--33 Yayoi-cho, Inage-ku, 263--8522, Chiba, Japan}
\temail{sirakawa@faculty.chiba-u.jp}
\vspace{-2ex}

\footcomment{
AMS Subject Classification: 
35K70, 
35G50, 
49J20, 
49K20. 
74N20. 
\\
Keywords: KWC system, pseudo-parabolic structure, optimal control problem, existence and semi-continuous dependence, 1st order necessary condition of optimality
}
\maketitle

\vspace{-4ex}

\noindent
{\bf Abstract.}
The KWC system is a well-known generic framework for phase-field models of grain boundary 
motion, whose original formulation is given as a parabolic gradient flow of a free energy. In the 
original KWC system, the results of uniqueness have been relatively scarce compared to other 
issues, such as existence, qualitative behavior, and numerics. This lack of progress has posed a 
significant challenge for more advanced topics, including optimal control. 
To overcome this, the authors have recently introduced the pseudo-parabolic structure to 
simultaneously preserve the gradient flow nature of free-energy, and to ensure the 
well-posedness including the uniqueness. The goal of this paper is to study an optimization
problem constrained by pseudo-parabolic KWC system. The theory will be developed 
through a series of Main Theorems concerning the existence and semi-continuous 
dependence of optimal controls, and first-order necessary conditions for optimality.

\pagebreak 

\section*{Introduction}

Let $ 0 < T < \infty $ be a fixed constant of time. Let $ N \in \{1, 2, 3, 4\} $ be a constant of dimension, and let $ \Omega \subset \R^N $ be a bounded domain with a Lipschitz boundary $ \Gamma := \partial \Omega $. Besides, we let $ Q := (0, T) \times \Omega $, $ \Sigma := (0, T) \times \Gamma $, $ H := L^2(\Omega) $, and $ \sH := L^2(0, T; H) $.

In this paper, we consider a class of optimal control problems $ \mbox{(OCP)}_\varepsilon $, parametrized by $ \varepsilon \in [0, 1] $. Each problem (OCP)$_\varepsilon$ is motivated by a constrained temperature control problem arising in grain boundary motion, and is formulated as follows.
\bigskip

\noindent
{\bf (OCP)$_\varepsilon$:} Find a pair of function $[u^*, v^*] \in [\sH]^2$, called the \textit{optimal control}, such that 
\begin{gather}
  [u^*, v^*] \in \sU_\ad := \{ [\tilde u, \tilde v] \in [\sH]^2 \,|\, \underline{u} \leq \tilde u \leq \overline{u} \mbox{ a.e.  in } Q \}, \mbox{ and } \label{ocp}
  \\
  \sJ_\varepsilon(u^*, v^*) = \min \{\sJ_\varepsilon(u,v) \,|\, [u,v] \in \sU_\ad\},
\end{gather}
where $\sJ_\varepsilon$ is a cost functional on $[\sH]^2$, defined as follows:
\begin{align}
  \sJ_\varepsilon&: [u,v] \in [\sH]^2 \mapsto \sJ_\varepsilon(u,v)
  \\
  &:= \frac{M_\eta}{2} \int_0^T |(\eta - \eta_{\rm ad})(t)|_H^2 \,dt + \frac{M_\theta}{2} \int_0^T |(\theta - \theta_\ad)(t)|_H^2 \,dt \label{cost}
  \\
  &\quad + \frac{M_u}{2} \int_0^T |u(t)|_H^2 \,dt + \frac{M_v}{2} \int_0^T |v(t)|_H^2 \,dt \in [0,\infty),
\end{align}
with $[\eta, \theta] \in [\sH]^2$ solving the following state system, denoted by (S)$_\varepsilon$.

\begin{align}
    \mbox{{\bf (S)}$_\varepsilon$}:&
    \\
    & \begin{cases}
        \partial_t \eta -\mathit{\Delta} \bigl( \eta +\mu^2 \partial_t \eta \bigr) +g(\eta) +\alpha'(\eta)\sqrt{\varepsilon^2 + |\nabla \theta|^2} = L_u u(t, x), ~ \mbox{for $ (t, x) \in Q $,}
        \\[0.5ex]
        \nabla \bigl( \eta +\mu^2 \partial_t \eta \bigr) \cdot n_\Gamma =  0, ~ \mbox{on $ \Sigma $,}
        \
        \\[0.5ex]
        \eta(0, x) = \eta_0(x),~ \mbox{for $ x \in \Omega $,}
    \end{cases}
    \label{(S)1st}
    \\[1.5ex]
    & \begin{cases}
        \ds \alpha_0(\eta) \partial_t \theta -\mathrm{div} \left( \alpha(\eta) \frac{\nabla \theta}{\sqrt{\varepsilon^2 + |\nabla \theta|^2}} +\nu^2 \nabla \partial_t \theta \right) = L_v v(t, x), ~ \mbox{for $ (t, x) \in Q $,}
        \\[1ex]
        \left( \alpha(\eta) \frac{\nabla \theta}{\sqrt{\varepsilon^2 + |\nabla \theta|^2}} +\nu^2 \nabla \partial_t \theta \right) \cdot n_\Gamma = 0 ~ \mbox{on $ \Sigma $,}
        \\[0.5ex]
        \theta(0, x) = \theta_0(x),~ \mbox{for $ x \in \Omega $.}
    \end{cases}
    \label{(S)2nd}
\end{align} 
For precise assumptions on the data, see section~\ref{s:sec2}.

The state system (S)$_\varepsilon$ is referred to as the \emph{KWC system}, which is known as a phase-field model of planar grain boundary motion, proposed by Kobayashi--Warren--Carter \cite{MR1752970,MR1794359}. 
In this context, the unknowns $ \eta = \eta(t, x) $ and $ \theta = \theta(t, x) $ are order parameters that indicate the \emph{orientation order} and \emph{orientation angle} of the polycrystal body, respectively. Moreover, $ \eta_0 = \eta_0(x) $ and $ \theta_0 = \theta_0(x) $ are the \emph{initial data} for $ \eta $ and $ \theta  $, respectively. The functions $ \alpha_0 = \alpha_0(\eta) $ and $ \alpha = \alpha(\eta) $ are fixed positive-valued functions to reproduce the mobilities of grain boundary motions. Additionally, $ g = g(\eta) $ is a perturbation for the orientation order $ \eta $, having a nonnegative potential $ G = G(\eta) $, i.e. $ \frac{d}{d\eta} G(\eta) = g(\eta) $. 
Finally, $ u = u(t, x) $ and $ v = v(t, x) $ serve as external forcing terms for $ \eta $ and $ \theta  $, respectively. 

In the optimal control problem (OCP)$_\varepsilon$, the forcing pair $ [u, v] $ is referred to as the \emph{control}, which serve as principal variable of the problem. Especially, $ -u $ denotes the relative temperature, and thus, $ u $ plays a role of temperature control in the formation of grain boundaries. Here, $L_u \geq 0$, $L_v \geq 0$, $M_\eta \geq 0$, $M_\theta \geq 0$, $M_u \geq 0$, and $M_v \geq 0$ are fixed positive constants, which decide the balance of weights between the cost functional $\sJ_\varepsilon$ and the state system (S)$_\varepsilon$. The constants $L_u$, $L_v$, $M_u$, and $M_v$ satisfy the following conditions:
\begin{align}
    \mbox{If } L_u M_u = 0 (\mbox{ resp. } L_v M_v = 0), \mbox{ then } M_u = L_u = 0 \, (\mbox{resp. } L_v = M_v = 0).
\end{align}

Recently, the mathematical methods for KWC system have been developed by a lot of mathematicians from various viewpoints (cf. \cite{
    MR3670006,MR4677072,MR4683700,MR1720614,MR2469586,MR2548486,MR3038131,MR3362773,MR2836555,MR2668289,MR3155454,MR3268865,MR4352617,
    MR1752970,MR1794359,
    MR4395725,MR4218112,MR3888633,MR4806756,
    MR4529388,mizuno2024weak,giga2024fractional,MR3937357,MR4603338,MR4809351
}). Although the systems considered in these previous studies differ in their specific formulations, they share a common origin in the form of the following evolution equation:
\begin{gather}\label{EE}
    -\mathcal{A}_0\bigl(\eta(t) \bigr) \frac{d}{dt} \left[ \begin{matrix}
        \eta(t) \\[1ex] \theta(t)
    \end{matrix} \right] = \frac{\delta}{\delta [\eta, \theta]} \mathcal{F}_\varepsilon(\eta(t), \theta(t)) +\left[ \begin{matrix}
        u(t) \\[1ex] v(t)
    \end{matrix}
    \right]~~ \mbox{in $ 
    [L^2(\Omega)]^2 $,}
    \\
    \mbox{for $ t \in (0, T) $.}
\end{gather}
Evolution equation \eqref{EE} is based on the \emph{gradient flow} of the following free-energy, called \emph{KWC energy}, which is defined as:
\begin{align}\label{FE}
    & \mathcal{F}_\varepsilon : [\eta, \theta] \in D(\mathcal{F}_\varepsilon) \subset H^1(\Omega) \times [BV(\Omega) \cap L^2(\Omega)]
    \nonumber
    \\
    & \mapsto \mathcal{F}_\varepsilon(\eta, \theta) := \frac{1}{2} \int_\Omega |\nabla \eta|^2 \, dx +\int_\Omega G(\eta) \, dx +\int_\Omega \alpha(\eta) \sqrt{\varepsilon^2 +|D \theta|^2} \in [0, \infty), 
\end{align}
with a functional derivative $ \frac{\delta}{\delta [\eta, \theta]} \mathcal{F}_\varepsilon $, and an unknown-dependent monotone operator $ \mathcal{A}_0(\eta) $ $\subset [L^2(\Omega)]^2 $. Indeed, our system (S) is derived from the following evolution equation \eqref{EE} in the case when:
\begin{gather}\label{A_mu^nu}
    \mathcal{A}_0(\eta) := \left[ \, \begin{matrix}
        I +\mu^2(-\mathit{\Delta}_0) & O
        \\[1ex]
        O & \alpha(\eta) I +\nu^2 (-\mathit{\Delta}_0)
    \end{matrix} \, \right] 
    \subset [L^2(\Omega)]^{2} \times [L^2(\Omega)]^{2},
\end{gather}
for each $ \eta \in L^2(\Omega) $, where $ I \subset L^2(\Omega) \times L^2(\Omega) $ and $ O \subset L^2(\Omega) \times L^2(\Omega) $ are the identity map and zero map, respectively, and $ \mathit{\Delta}_0 \subset L^2(\Omega) \times L^2(\Omega) $ is the Laplacian subject to the zero-Neumann boundary condition.

The original KWC system corresponds to the case when $ \mu = \nu = 0 $, and is formulated as a nonlinear parabolic system. In the original setting, most mathematical studies have focused on the existence and qualitative behavior of solutions (cf. \cite{MR3670006,MR4677072,MR4683700,MR1720614,MR2469586,MR2548486,MR3038131,MR3362773,MR2836555,MR2668289,MR3155454,MR3268865,MR4352617}), numerical simulations (cf. \cite{MR1752970,MR1794359,MR3951294}), and related topics (cf. \cite{MR4395725,MR4218112,MR3888633,MR4806756,
MR4529388,giga2024fractional,MR3937357,MR4603338,MR4809351}). In contrast, results concerning the uniqueness of solutions have been relatively scarce (cf. \cite{MR2469586,MR3155454,MR4218112,MR4395725}). A key difficulty in addressing uniqueness is that the operator $ \mathcal{A}_0(\eta) $ depends on the unknown $ \eta $, while standard parabolic regularity for $ \eta $ is insufficient to resolve the issue in general higher-dimensional settings of $ \Omega $. Indeed, uniqueness results have only been obtained in special cases, such as when the spatial domain is one-dimensional (cf. \cite[Theorem 2.2]{MR2469586}), or when the operator $ \mathcal{A}_0 $ is independent of $ \eta $ (cf. \cite[Main Theorem 1]{MR4218112}, \cite[Main Theorem 1]{MR4395725}, and \cite[Theorem 2.2]{MR3155454}). Outside of these cases, the uniqueness issue has remained an open problem for many years. This lack of progress has also presented a significant challenge to further developments in advanced topics, including optimization problems.

Physically, the operator $ \mathcal{A}_0(\eta) $ represents the mobility (driving force) of grain boundary motion. In the KWC model, grain boundaries are assumed to form in regions where the orientation order $ \eta $ remains undeveloped. In this context, the assumption that the mobility depends dynamically on $ \eta $ reflects a more realistic representation of the underlying physical phenomena. Therefore, the possibility of advancing mathematical analysis under this more physically faithful setting plays a crucial role in determining the practical applicability of the results.

In the formula \eqref{EE} and \eqref{A_mu^nu}, the pseudo-parabolic structure provided by $ (-\mathit{\Delta}_0) $ has been proposed to preserve the gradient flow nature of the KWC system, and to ensure sufficient regularity for establishing uniqueness, simultaneously. Indeed, recent studies such as \cite{mizuno2024weak,MR4683693,MR4797677} have demonstrated that the pseudo-parabolic formulation leads to a mathematically well-posed version of the KWC system. This development indicates that incorporating pseudo-parabolic feature provides a solid foundation for advancing the study of optimal control problems for grain boundary motion, even in the unknown dependent setting  of the mobility.

Consequently, we set the goal to establish a mathematical theory of optimization governed by the pseudo-parabolic KWC system (S)$_\varepsilon$. The mathematical theory is developed through the following Main Theorems.
\begin{description}
    \item[Main Theorem 1:]Existence of the optimal control for (OCP)$_\varepsilon$.
        \vspace{-1ex}
    \item[Main Theorem 2:]Semi-continuous dependence of (OCP)$_\varepsilon$ with respect to $ \varepsilon $ and the initial data of (S)$_\varepsilon$.
        \vspace{-1ex}
    \item[Main Theorem 3:]First-order necessary optimality conditions when $ \varepsilon > 0 $.
        \vspace{-1ex}
    \item[Main Theorem 4:]Necessary condition for optimality when $ \varepsilon = 0 $.
\end{description}

This paper is organized as follows. Preliminaries are given in Section 1, and on this basis, the Main Theorems are stated in Section 2. For the proofs of Main Theorems, some Key-Lemmas are prepared in Section 3 . Based on these, the Main Theorems are proved in Section 4.

\section{Preliminaries}
We begin by prescribing the notations used throughout this paper. 
\smallskip
As notations of base spaces, we define:
\begin{equation*}
    H := L^2(\Omega), \ V := H^1(\Omega), ~ 
    \sH := L^2(0,T;H),  \mbox{ and } \sV := L^2(0,T;V),
\end{equation*}
and we suppose that:
\begin{equation}
    V \subset H = H^* \subset V^*, \mbox{ and } \sV \subset \sH = \sH^* \subset \sV^*,
\end{equation}
where ``$=$'' is due to the identifications of the Hilbert spaces with their duals, and ``$\subset$'' implies continuous embeddings.
\smallskip

\noindent
\underline{\textbf{\textit{Notations in real analysis.}}}
We define:
\begin{align*}
    & r \vee s := \max \{ r, s \} ~ \mbox{ and } ~ r \wedge s := \min \{r, s\}, \mbox{ for all $ r, s \in [-\infty, \infty] $,}
\end{align*}
and especially, we write:
\begin{align*}
    & [r]^+ := r \vee 0 ~ \mbox{ and } ~ [r]^- := -(r \wedge 0), \mbox{ for all $ r \in [-\infty, \infty] $.}
\end{align*}

Let $ d \in \N $ be a fixed dimension. We denote by $ |y| $ and $ y \cdot z $ the Euclidean norm of $ y \in \mathbb{R}^d $ and the scalar product of $ y, z \in \R^d $, respectively, i.e., 
\begin{equation*}
\begin{array}{c}
| y | := \sqrt{y_1^2 +\cdots +y_d^2} \mbox{ \ and \ } y \cdot z  := y_1 z_1 +\cdots +y_d z_d, 
\\[1ex]
\mbox{ for all $ y = [y_1, \ldots, y_d], ~ z = [z_1, \ldots, z_d] \in \mathbb{R}^d $.}
\end{array}
\end{equation*}
Besides, we let:
\begin{align*}
    & \mathbb{B}^d := \left\{ \begin{array}{l|l}
        y \in \R^d & |y| < 1
    \end{array} \right\} ~ \mbox{ and } ~ \mathbb{S}^{d -1} := \left\{ \begin{array}{l|l}
        y \in \R^d & |y| = 1
    \end{array} \right\}.
\end{align*}
We denote by $\mathcal{L}^{d}$ the $ d $-dimensional Lebesgue measure, and we denote by $ \mathcal{H}^{d} $ the $ d $-dimensional Hausdorff measure.  In particular, the measure theoretical phrases, such as ``a.e.'', ``$dt$'', and ``$dx$'', and so on, are all with respect to the Lebesgue measure in each corresponding dimension. Also on a Lipschitz-surface $ S $, the phrase ``a.e.'' is with respect to the Hausdorff measure in each corresponding Hausdorff dimension. In particular, if $S$ is $C^1$-surface, then we simply denote by $dS$ the area-element of the integration on $S$.

For a Borel set $ E \subset \R^d $, we denote by $ \chi_E : \R^d \longrightarrow \{0, 1\} $ the characteristic function of $ E $. Additionally, for a distribution $ w $ on an open set in $ \R^d $ and any $i \in \{ 1,\dots,d \}$, let $ \partial_i w$ be the distributional differential with respect to $i$-th variable of $w$. As well as we consider, the differential operators, such as $\nabla,\ \diver, \ \nabla^2$, and so on, are considered in distributional senses.
\bigskip

\noindent
\underline{\textbf{\textit{Abstract notations. (cf. \cite[Chapter II]{MR0348562})}}}
For an abstract Banach space $ X $, we denote by $ |\cdot|_{X} $ the norm of $ X $, and denote by $ \langle \cdot, \cdot \rangle_X $ the duality pairing between $ X $ and its dual $ X^* $. In particular, when $ X $ is a Hilbert space, we denote by $ (\cdot,\cdot)_{X} $ the inner product of $ X $. 

For two Banach spaces $ X $ and $ Y $,  let $  \mathscr{L}(X; Y)$ be the Banach space of bounded linear operators from $ X $ into $ Y $. 

For Banach spaces $ X_1, \dots, X_d $ with $ 1 < d \in \N $, let $ X_1 \times \dots \times X_d $ be the product Banach space endowed with the norm $ |\cdot|_{X_1 \times \cdots \times X_d} := |\cdot|_{X_1} + \cdots +|\cdot|_{X_d} $. However, when all $ X_1, \dots, X_d $ are Hilbert spaces, $ X_1 \times \dots \times X_d $ denotes the product Hilbert space endowed with the inner product $ (\cdot, \cdot)_{X_1 \times \cdots \times X_d} := (\cdot, \cdot)_{X_1} + \cdots +(\cdot, \cdot)_{X_d} $ and the norm $ |\cdot|_{X_1 \times \cdots \times X_d} := \bigl( |\cdot|_{X_1}^2 + \cdots +|\cdot|_{X_d}^2 \bigr)^{\frac{1}{2}} $. In particular, when all $ X_1, \dots,  X_d $ coincide with a Banach space $ Y $, the product space $X_1 \times \dots \times X_d$ is simply denoted by $[Y]^d$.
\bigskip

\noindent
\underline{\textbf{\textit{Notations of basic differential operators.}}}
Let $\Lap_0$ be the Laplacian operator subject to zero-Neumann boundary condition. We can regard $\Lap_0$ as a bounded linear operator from $V$ to $V^*$, with the following variational identity:
\begin{equation}
    \langle \Lap_0 w_1, w_2 \rangle_V = -(\nabla w_1, \nabla w_2)_{[H]^N}, \ \mbox{ for } w_1, w_2 \in V.
\end{equation}

In the meantime, we can regard the distributional divergence ``$\diver$'' as a bounded linear operator from $[H]^N$ to $V^*$, given as:
\begin{equation}
    \langle \diver {\bm w}, \varphi \rangle_V := -({\bm w}, \nabla \varphi)_{[H]^N}, \ \mbox{ for every } {\bm w} \in [H]^N, \mbox{ and } \varphi \in V.
\end{equation}

\smallskip

\noindent
\underline{\textbf{\textit{Notations for the time-discretization.}}}
Let $\tau > 0$ be a constant of the time step-size, and let $\{ t_i \}_{i=0}^\infty \subset [0,\infty)$ be the time sequence defined as:
\begin{equation*}
  t_i := i\tau,\ i=0,1,2,\ldots.
\end{equation*}
Let $X$ be a Banach space. Then, for any sequence $\{ [t_i,w_i] \}_{i=0}^\infty \subset[0,\infty) \times X$, we define the \textit{forward time-interpolation} $[\overline{w}]_\tau \in L^\infty_\mathrm{loc}([0,\infty); X)$, the \textit{backward time-interpolation} $[\underline{w}]_\tau \in L^\infty_\mathrm{loc}([0,\infty);X)$ and the \textit{linear time-interpolation} $[w]_\tau \in W^{1,2}_\mathrm{loc}([0,\infty);X)$, by letting:
\begin{equation*}
  \left\{\begin{aligned}
    &[\overline{w}]_\tau(t) := \chi_{(-\infty,0]} z_0 + \sum_{i=1}^\infty \chi_{(t_{i-1}, t_i]}(t) w_i, \\
    &[\underline{w}]_\tau(t) := \sum_{i=0}^\infty \chi_{(t_{i}, t_{i+1}]}(t) w_{i}, \\
    &[w]_\tau (t) := \sum_{i=1}^\infty \chi_{[t_{i-1},t_i)}(t) \left(\frac{t-t_{i-1}}{\tau} w_{i} + \frac{t_i - t}{\tau} w_{i-1}\right),
  \end{aligned}\right. ~{\rm in}~ X,\ {\rm for}~ t \geq 0, \label{eq:tI}
\end{equation*}
respectively.

For a constant $q \in [1 ,\infty]$ and for an open interval $I \subset \R$, we can say that $L^q(I;X) \subset L_{\rm loc}^q(\R;X)$ by identifying an $X$-valued function on $I$ with its zero-extension onto $\R$. Also, the following facts hold.
\smallskip

\noindent
$\bullet$ For any $ q \in [1, \infty) $ and any $ w \in L^q(0,T; X) $, we denote by $ \{ w_i \}_{i = 1}^\infty \subset X $ the sequence of time-discretization data of $ w $, defined as: 
\begin{subequations}\label{tI}
\begin{align}\label{tI01}
    &w_i := \frac{1}{\tau} \int_{t_{i -1}}^{t_i} w(\varsigma) \, d \varsigma ~ \mbox{ in $ X $, ~ for $ i = 1, 2, 3, \dots $.}
\end{align}
As is easily checked, the forward time-interpolation $ [\overline{w}]_\tau \in L^q_\mathrm{loc}([0, \infty); X) $ for the above $ \{ w_i \}_{i = 1}^\infty $ fulfills that:
\begin{align}
    & [\overline{w}]_\tau \to w \mbox{ in $ L^q(0,T; X) $, as $ \tau \downarrow 0 $.}\label{tI02}
\end{align}
Moreover, if $w \in L^\infty(Q)$, then,
\begin{gather}
    |[\overline{w}]_\tau(t)|_{L^\infty(\Omega)} \leq |w|_{L^\infty(Q)}, \mbox{ for any } t \in [0,T] \mbox{ and any } \tau > 0, \label{tI03}
\end{gather}
and
\begin{gather}
    [\overline{w}]_\tau \to w \mbox{ weakly-$*$ in } L^\infty(Q) \mbox{ and in the pointwise sense a.e. in $Q$, } \mbox{ as } \tau \downarrow 0. \label{tI04}
\end{gather}

\noindent
$\bullet$ For any $w \in W^{1,2}(0,T;X)$, the sequence $\{ w_i \}_{i=0}^\infty \subset X$ is given as:
\begin{equation}
    w_i := \begin{cases}
        w(t_i), & \mbox{ if } t_i \leq T,
        \\
        w(t_{i-1}), & \mbox{ if } t_{i-1} \leq T < t_i,
        \\
        0, &\mbox{otherwise}.
    \end{cases} \label{tI05}
\end{equation}
Then, the following convergences hold as $\tau \downarrow 0$:
\begin{equation}
    [w]_\tau \to w \mbox{ in } C([0,T];X) \mbox{ and weakly in } W^{1,2}(0,T;X),
\end{equation}
and
\begin{equation}
    [\overline{w}]_\tau \to w, \ [\underline{w}]_\tau \to w \mbox{ in } L^\infty(0,T;H). \label{tI06}
\end{equation}
\end{subequations}
\smallskip

\noindent
\underline{\textbf{\textit{Notations in convex analysis.}}}
Let $X$ be an abstract Hilbert space. Then, for any closed and convex set $K \subset X$, we can define a single-valued operator $\proj_K:X \longrightarrow K$, which maps $w \in X$ to a point $\proj_K(w) \in K$, satisfying:
\begin{equation}
    |\proj_K(w) - w|_X = \min \{ |\tilde w - w|_X \,|\, \tilde w \in K \}.
\end{equation}

\begin{rem}
    Let $K$ be a closed and convex set in a Hilbert space $X$. Then, the projection $\proj_K$ satisfies the following fact:
    \begin{equation}
        \tilde w = \proj_K(w) \mbox{ if and only if } (w - \tilde w, z - \tilde w)_X \leq 0, \ \mbox{ for any } z \in K. \label{proj01}
    \end{equation}
    The operator $\proj_K$ is called the \textit{orthogonal projection} (or \textit{projection} in short) onto $K$.
\end{rem}

\begin{rem} \label{constraint}
    We define $\sK \subset \sH$ as the constraint with the obstacles $\underline{u}, \overline{u} \in L^\infty(Q)$, i.e.,
    \begin{equation}
        \sK := \{ \tilde u \in \sH \,| \, \underline{u} \leq \tilde u \leq \overline{u}, \mbox{ a.e. in } Q \}. 
    \end{equation}
    Then, the projection $\proj_\sK: \sH \longrightarrow \sK$ is given by:
    \begin{gather}
        \begin{aligned}
            [\proj_\sK(u)](t,x) &= \underline{u}(t,x) \vee (\overline{u}(t,x) \wedge u(t,x))
            \\
            &= \begin{cases}
                \overline{u}(t,x), &\mbox{ if } u(t,x) > \overline{u}(t,x),
                \\
                u(t,x), & \mbox{ if } \underline{u}(t,x) \leq u(t,x) \leq \overline{u}(t,x),
                \\
                \underline{u}(t,x), & \mbox{ if } u(t,x) < \underline{u}(t,x),
            \end{cases}
        \end{aligned}
        \\
        \mbox{ a.e. } (t,x) \in Q, \ \mbox{ and for any } u \in \sH.
    \end{gather}
\end{rem}

For a proper, lower semi-continuous (l.s.c.), and convex function $\Psi : \,X \longrightarrow (-\infty, \infty]$ on a Hilbert space $X$, we denote by $D(\Psi)$ the effective domain of $\Psi$. Also, we denote by $\partial \Psi$ the subdifferential of $\Psi$. The set $D(\partial \Psi) := \left\{ z \in X\,|\, \partial \Psi(z) \neq \emptyset \right\}$ is called the domain of $\partial\Psi$. The subdifferential $ \partial \Psi $ is known as a maximal monotone graph in the product space $X \times X$. We often use the notation ``$[z_0, z_0^*] \in \partial \Psi ~{\rm in}~ X \times X$", to mean that ``$z_0^* \in \partial \Psi(z_0) ~{\rm in}~ X~{\rm for}~ z_0 \in D(\partial \Psi)$", by identifying the operator $\partial \Psi$ with its graph in $X\times X$.
\medskip
\begin{ex}
    \label{exConvex}
    Let $\{ \gamma_\varepsilon \}_{\varepsilon \geq 0}$ be a class of convex functions on $\R^N$, defined as:
    \begin{equation}
        \gamma_\varepsilon: y \in \R^N \mapsto \gamma_\varepsilon(y) := \sqrt{\varepsilon^2 + |y|^2} \in [0,\infty).
    \end{equation}
    Then, the following items hold. 
    \begin{description}
        \item[(O)]The subdifferential $ \partial \gamma_0 \subset \R^N \times \R^N $ of the convex function $ \gamma_0 : y = [y_1, \dots, y_n] \in \R^N \mapsto |y| =  \sqrt{y_1^2 + \dots +y_N^2} \in [0, \infty) $ coincides with the following set-valued function $ \Sgn: \R^N \rightarrow 2^{\mathbb{R}^N} $, which is defined as:
\begin{align}\label{Sgn^d}
\Sgn :  y = [y_1, & \dots, y_N] \in \mathbb{R}^N \mapsto \Sgn(y) = \Sgn(y_1, \dots, y_N) 
  \nonumber
  \\
  & := \left\{ \begin{array}{ll}
          \multicolumn{2}{l}{
                  \ds \frac{y}{|y|} = \frac{[y_1, \dots, y_N]}{\sqrt{y_1^2 +\cdots +y_N^2}}, ~ } \mbox{if $ y \ne 0 $,}
                  \\[3ex]
          \overline{\mathbb{B}^N}, & \mbox{otherwise.}
      \end{array} \right.
  \end{align}
\item[(\,I\,)]For every $ \varepsilon > 0 $, the subdifferential $\partial \gamma_\varepsilon$ is identified with the (single-valued) usual gradient, i.e.:
\begin{equation}
    D(\partial \gamma_\varepsilon) = \R^N ~and~\nabla \gamma_\varepsilon : \R^N \ni y \mapsto \nabla \gamma_\varepsilon(y) := \frac{y}{\sqrt{\varepsilon^2 + |y|^2}} \in \R^N.
\end{equation}
Moreover, since:
    \begin{align*}
        \gamma_\varepsilon(y) = \bigl| [\varepsilon, y] \bigr|_{\R^{N +1}} & =\bigl| [\varepsilon, y_1, \dots, y_N] \bigr|_{\R^{N +1}}, \mbox{ for all $ [\varepsilon, y] = [\varepsilon, y_1, \dots, y_N] \in \R^{N +1} $,}
        \\
        & \mbox{with $ \varepsilon \geq 0 $ and $ y = [y_1, \dots, y_N] \in \R^N $,}
    \end{align*}
    it will be estimated that:
    \begin{equation}
      \begin{aligned}
        & 
        \begin{cases} 
            \ds \bigl| \nabla \gamma_\varepsilon(y) \bigr|_{\R^{N}} = \left| \frac{y}{\bigl| [\varepsilon, y] \bigr|_{\mathbb{R}^{N +1}}} \right|_{\R^{N}} \leq \left| \frac{[\varepsilon, y]}{\bigl| [\varepsilon, y] \bigr|_{\mathbb{R}^{N +1}}} \right|_{\R^{N +1}} = 1,
            \\[3ex]
        \ds 
            |\partial_i \partial_j \gamma_\varepsilon(y)| \leq \frac{1}{\varepsilon}, ~ \mbox{ for all $ \varepsilon > 0 $, $ y \in \R^N $, and $ i, j = 1, \dots, N $.}
        \end{cases}
    \end{aligned}\label{exM01}
    \end{equation}
\end{description}
\end{ex}

\begin{ex} \label{exConvex2}
    For any $\varepsilon \geq 0$ and open interval $I \subset (0,T)$, we define a convex function $\Phi_\varepsilon^I$ on $[\sH]^N$, as follows:
    \begin{equation}
        \Phi_\varepsilon^I:{\bm w} \in [\sH]^N \mapsto \Phi_\varepsilon^I({\bm w}) := \int_I \int_\Omega \gamma_\varepsilon({\bm w}) \,dxdt \in [0,\infty),
    \end{equation}
    where $\gamma_\varepsilon$ is as in Example \ref{exConvex}. Then, the subdifferential $\partial \Phi_\varepsilon^I({\bm w}) \subset [\sH]^N$ at ${\bm w} \in [\sH]^N$ is given by:
    \begin{equation}
        \partial \Phi_\varepsilon^I({\bm w}) = \{ {\bm z} \in [\sH]^N \,| \,{\bm z} \in \partial \gamma_\varepsilon({\bm w}), \mbox{ a.e. in } Q \}. 
    \end{equation}
\end{ex}

Finally, we mention about a notion of convergence of convex functions, known as ``Mosco-convergence''.

\begin{defn}[Mosco-convergence: cf. \cite{MR0298508}]\label{Def.Mosco}
    Let $ X $ be an abstract Hilbert space. Let $ \Psi : X \rightarrow (-\infty, \infty] $ be a proper, l.s.c., and convex function, and let $ \{ \Psi_n \}_{n = 1}^\infty $ be a sequence of proper, l.s.c., and convex functions $ \Psi_n : X \rightarrow (-\infty, \infty] $, $ n = 1, 2, 3, \dots $.  Then, it is said that $ \Psi_n \to \Psi $ on $ X $, in the sense of Mosco, as $ n \to \infty $, iff. the following two conditions are fulfilled:
    \begin{description}
      \item[(\hypertarget{M_lb}{M1}) The condition of lower-bound:]$ \ds \varliminf_{n \to \infty} \Psi_n(\check{w}_n) \geq \Psi(\check{w}) $, if $ \check{w} \in X $, $ \{ \check{w}_n  \}_{n = 1}^\infty \subset X $, and $ \check{w}_n \to \check{w} $ weakly in $ X $, as $ n \to \infty $. 
      \item[(\hypertarget{M_opt}{M2}) The condition of optimality:]for any $ \hat{w} \in D(\Psi) $, there exists a sequence \linebreak $ \{ \hat{w}_n \}_{n = 1}^\infty  \subset X $ such that $ \hat{w}_n \to \hat{w} $ in $ X $ and $ \Psi_n(\hat{w}_n) \to \Psi(\hat{w}) $, as $ n \to \infty $.
    \end{description}
    As well as, if the sequence of convex functions $ \{ \widehat{\Psi}_\varepsilon \}_{\varepsilon \in \Xi} $ is labeled by a continuous argument $\varepsilon \in \Xi$ with a range $\Xi \subset \mathbb{R}$ , then for any $\varepsilon_{0} \in \Xi$, the Mosco-convergence of $\{ \widehat{\Psi}_\varepsilon \}_{\varepsilon \in \Xi}$, as $\varepsilon \to \varepsilon_{0}$, is defined by those of subsequences $ \{ \widehat{\Psi}_{\varepsilon_n} \}_{n = 1}^\infty $, for all sequences $\{ \varepsilon_n \}_{n=1}^{\infty} \subset \Xi$, satisfying $\varepsilon_{n} \to \varepsilon_{0}$ as $n \to \infty$.
  \end{defn}
  
  \begin{rem}\label{Rem.MG}
    Let $ X $, $ \Psi $, and $ \{ \Psi_n \}_{n = 1}^\infty $ be as in Definition~\ref{Def.Mosco}. Then, the following hold.
    \begin{description}
      \item[(\hypertarget{Fact1}{Fact\,1})](cf. \cite[Theorem 3.66]{MR0773850} and \cite[Chapter 2]{Kenmochi81}) Let us assume that
      \begin{equation}\label{Mosco01}
        \Psi_n \to \Psi \mbox{ on $ X $, in the sense of  Mosco, as $ n \to \infty $,}
        \vspace{-1ex}
      \end{equation}
  and
  \begin{equation*}
  \left\{ ~ \parbox{10cm}{
  $ [w, w^*] \in X \times X $, ~ $ [w_n, w_n^*] \in \partial \Psi_n $ in $ X \times X $, $ n \in \N $,
  \\[1ex]
  $ w_n \to w $ in $ X $ and $ w_n^* \to w^* $ weakly in $ X $, as $ n \to \infty $.
  } \right.
  \end{equation*}
  Then, it holds that:
  \begin{equation*}
  [w, w^*] \in \partial \Psi \mbox{ in $ X \times X $, and } \Psi_n(w_n) \to \Psi(w) \mbox{, as $ n \to \infty $.}
  \end{equation*}
      \item[(\hypertarget{Fact2}{Fact\,2})](cf. \cite[Lemma 4.1]{MR3661429} and \cite[Appendix]{MR2096945}) Let $ d \in \mathbb{N} $ denote dimension constant, and let $  S \subset \R^d $ be a bounded open set. Then, under the Mosco-convergence as in \eqref{Mosco01}, a sequence $ \{ \widehat{\Psi}_n^S \}_{n = 1}^\infty $ of proper, l.s.c., and convex functions on $ L^2(S; X) $, defined as:
          \begin{equation*}
              w \in L^2(S; X) \mapsto \widehat{\Psi}_n^S(w) := \left\{ \begin{array}{ll}
                      \multicolumn{2}{l}{\ds \int_S \Psi_n(w(t)) \, dt,}
                      \\[1ex]
                      & \mbox{ if $ \Psi_n(w) \in L^1(S) $,}
                      \\[2.5ex]
                      \infty, & \mbox{ otherwise,}
                  \end{array} \right. \mbox{for $ n = 1, 2, 3, \dots $;}
          \end{equation*}
          converges to a proper, l.s.c., and convex function $ \widehat{\Psi}^S $ on $ L^2(S; X) $, defined as:
          \begin{equation*}
              z \in L^2(S; X) \mapsto \widehat{\Psi}^S(z) := \left\{ \begin{array}{ll}
                      \multicolumn{2}{l}{\ds \int_S \Psi(z(t)) \, dt, \mbox{ if $ \Psi(z) \in L^1(S) $,}}
                      \\[2ex]
                      \infty, & \mbox{ otherwise;}
                  \end{array} \right. 
          \end{equation*}
          on $ L^2(S; X) $, in the sense of Mosco, as $ n \to \infty $. 
  \end{description}
  \end{rem}
  
  \begin{ex}[Examples of Mosco-convergence]\label{Rem.ExMG}
      Let $ \varepsilon_0 \geq 0 $ be arbitrary fixed constant.
      \begin{description}
        \item[(O)] $\{ \gamma_\varepsilon \}_{\varepsilon \geq 0}$ be as in Example \ref{exConvex}. Then, it holds that:
      \begin{equation}
        \gamma_\varepsilon \to \gamma_{\varepsilon_0} \mbox{ on $ \R^N $, in the sense of Mosco, as $ \varepsilon \to \varepsilon_0 $.}
      \end{equation}
      \item[(\,I\,)] Let $I \subset (0,T)$ be an arbitrary open interval, and let $\{ \Phi_\varepsilon^I \}_{\varepsilon \geq 0}$ be the class of convex functions as in Example \ref{exConvex2}. According to the statement \textbf{(O)} and Fact 2 in Remark \ref{Rem.MG}, it holds that:
      \begin{equation}
        \Phi_\varepsilon^I \to \Phi_{\varepsilon_0}^I \mbox{ on } [\sH]^N, \mbox{ in the sense of Mosco, as } \varepsilon \to \varepsilon_0.
      \end{equation}
      \end{description}
  \end{ex}

\section{Main Theorem}\label{s:sec2}

In this paper, the main assertions are discussed under the following assumptions.

\begin{itemize}
  \item[(A0)] $0 < T < \infty$ is a fixed constant of time, and $N \in \{ 1,2,3,4 \}$ is a fixed dimension. $\Omega \subset \R^N$ is a bounded domain with a boundary $\Gamma := \partial\Omega$, and when $N > 1$, $\Gamma$ is a Lipschitz boundary with the unit outer normal $n_\Gamma$. 
  \item[(A1)] $\mu > 0$ and $\nu > 0$ are fixed constants. $M_\eta \geq 0$, $M_\theta \geq 0$ are fixed constants. $L_u \geq 0$, $L_v \geq 0$, $M_u \geq 0$ and $M_v \geq 0$ are fixed constants, such that:
  \begin{align}
    \mbox{If }L_u M_u = 0 \ (\mbox{resp. } L_v M_v = 0), \mbox{ then } M_u = L_u = 0 \ (\mbox{resp. } L_v = M_v = 0).
\end{align}
  Besides, the target profiles $\eta_\ad \in H$ and $\theta_\ad \in H$ are given functions. 
  \item[(A2)] $g: \R \longrightarrow \R$ is a $C^1$-class function with a nonnegative primitive $G \in C^2(\R)$. Moreover, $g$ satisfies the following condition:
  \begin{equation}
    \liminf_{\xi \downarrow -\infty} g(\xi) = -\infty, \ \limsup_{\xi \uparrow \infty} g(\xi) = \infty.
  \end{equation}
  \item[(A3)] $\alpha_0: \R \longrightarrow (0,\infty)$ is a $C^1$-class function, and $\alpha : \R \longrightarrow [0,\infty)$ is a $C^2$-class convex function, such that:
      \begin{equation}
        \alpha'(0) = 0, \, \alpha'' \geq 0 \mbox{ on } \R, \ \mbox{ and } \delta_* := \inf \alpha_0(\R) > 0.
      \end{equation}
  \item[(A4)] The initial data $[\eta_0, \theta_0]$ belongs to the class $[V \cap L^\infty(\Omega)] \times V$.
  \item[(A5)] Let $\overline{u}, \underline{u} \in L^\infty(Q)$ be obstacles and let $\sK$ be the constraint as in Remark \ref{constraint}, which provide a class of admissible controls $\sU_\ad$, defined by
  \begin{equation}
    \sU_\ad := \{ [\tilde u,\tilde v] \in [\sH]^2 \,|\, \tilde u \in \sK \}.
  \end{equation}
\end{itemize}

Now, the main results are stated as follows.
\begin{mTh}[Existence of optimal control] \label{mth1}
  Under the assumptions (A0)--(A5), the problem (OCP)$_\varepsilon$ has at least one optimal control $[u_\varepsilon^*, v_\varepsilon^*] \in \sU_\ad$.
\end{mTh}

\begin{mTh}[Semi-continuous dependence with respect to $\varepsilon$ and initial data] \label{mth2}
  Let us assume the assumptions (A0)--(A5). In addition, we suppose that sequences $\{ \varepsilon_n \}_{n \in \N} \subset [0,\infty)$, $\{[\eta_{0,n}, \theta_{0,n}]\}_{n \in \N} \subset [V]^2$ satisfy the following convergences:
  \begin{equation}
    \varepsilon_n \to \varepsilon, \mbox{ and } [\eta_{0,n}, \theta_{0,n}] \to [\eta_0, \theta_0] \mbox{ in } [V]^2 \mbox{ as } n \to \infty.
  \end{equation}
  Let $[u_n^*, v_n^*] \in \sU_\ad$ be the optimal control corresponding to $\varepsilon_n$ and the initial data $[\eta_{0,n}, \theta_{0,n}]$. Then, there exists a subsequence $\{ n_j \}_{j \in \N} \subset \{ n \}$, and a pair of functions $[u^*, v^*]$ such that
  \begin{equation}
    \left\{ \begin{aligned}
      &\bullet \ [\sqrt{M_u} u_{n_j}^*, \sqrt{M_v} v_{n_j}^*] \to [\sqrt{M_u} u^*, \sqrt{M_v} v^*] \mbox{ weakly in } [\sH]^2 \mbox{ as } j \to \infty,
      \\
      &\bullet \ [u^*, v^*] \mbox{ is an optimal control of (OCP)$_\varepsilon$}.
    \end{aligned} \right.
  \end{equation}
\end{mTh}

\begin{mTh}[First-order necessary conditions for optimality when $\varepsilon > 0$] \label{mth3}
  Let $\varepsilon > 0$, $[u^*, v^*]$ be the optimal control of (OCP)$_\varepsilon$, and $[\eta_\varepsilon^*, \theta_\varepsilon^*]$ be the solution to the state system (S)$_\varepsilon$ corresponding to the initial data $[\eta_0, \theta_0]$ and to the forcings $[u^*, v^*]$. Then, under the assumptions (A0)--(A5), for the optimal control $[u^*, v^*]$, it holds that:
  \begin{subequations}\label{NC0}
    \begin{equation}
      (L_u p_\varepsilon^* + M_u u^*, h-u^*)_\sH \geq 0, \ \mbox{ for any } h \in \sH, \label{NC0a}
    \end{equation}
    and
    \begin{equation}
      L_v z_\varepsilon^* + M_v v^* = 0 \mbox{ in } \sH. \label{NC0b}
    \end{equation}
  \end{subequations}
  In this context, the pair of functions $[p_\varepsilon^*, z_\varepsilon^*]$ solves the following variational system:
  \begin{gather}
    \bigl(\bigl( -\partial_t p_\varepsilon^* + g'(\eta_\varepsilon^*)p_\varepsilon^* + \alpha''(\eta_\varepsilon^*) \gamma_\varepsilon(\nabla \theta_\varepsilon^*) p_\varepsilon^* + \alpha_0'(\eta_\varepsilon^*) \partial_t \theta_\varepsilon^* z_\varepsilon^* + \alpha'(\eta_\varepsilon^*) \nabla \gamma_\varepsilon(\nabla \theta_\varepsilon^*) \cdot \nabla z_\varepsilon^*\bigr)(t), \varphi\bigr)_H 
    \\
    + (\nabla (p_\varepsilon^* - \mu^2 \partial_t p_\varepsilon^*)(t), \nabla \varphi)_{[H]^N} = (M_\eta(\eta_\varepsilon^* - \eta_\ad)(t), \varphi)_H, \label{adj_1}
    \\
    \mbox{ for any } \varphi \in V, \mbox{ and a.e. } t \in (0,T).
  \end{gather}
  and
  \begin{gather}
    \bigl(\bigl( -\alpha_0(\eta_\varepsilon^*) \partial_t z_\varepsilon^* - \alpha_0'(\eta_\varepsilon^*) \partial_t \eta_\varepsilon^* z_\varepsilon^* \bigr)(t), \psi\bigr)_H + (\alpha(\eta_\varepsilon^*(t)) \nabla^2 \gamma_\varepsilon(\nabla \theta_\varepsilon^*(t)) \nabla z_\varepsilon^*(t), \nabla \psi)_{[H]^N}
    \\
    + \bigl((- \nu^2 \nabla \partial_t z_\varepsilon^* + \alpha'(\eta_\varepsilon^*) p_\varepsilon^*\nabla\gamma_\varepsilon(\nabla \theta_\varepsilon^*))(t) , \nabla \psi\bigr)_{[H]^N} = (M_\theta(\theta_\varepsilon^* - \theta_\ad)(t), \psi)_H,\label{adj_2}
    \\
     \mbox{ for any } \psi \in V, \mbox{ and a.e. } t \in (0,T).
  \end{gather}
  with a terminal condition:
  \begin{equation}
    p_\varepsilon^*(T) = 0, \ z_\varepsilon^*(T) = 0, \ \mbox{ in } H. \label{adj_3}
  \end{equation}
\end{mTh}

\begin{mTh}[Necessary condition for optimality when $\varepsilon = 0$] \label{mth4}
  Let $\sW$ be a closed linear subspace of $\sH$, defined as:
  \begin{equation}
    \sW := \{ w \in W^{1,2}(0,T;V) \,|\, w(0) = 0 \mbox{ in } V \}.
  \end{equation}
  Under the assumptions (A0)--(A5), there exist an optimal control $[u^\circ, v^\circ] \in \sU_\ad$ of (OCP)$_0$ with the unique solution $[\eta^\circ, \theta^\circ] \in [\sH]^2$ to (S)$_0$ for the initial data $[\eta_0, \theta_0]$ and forcing pair $[u^\circ, v^\circ]$, a pair of functions $[p^\circ, z^\circ] \in [\sH]^2$, a triplet of functions $[\xi^\circ, \sigma^\circ, \varpi^\circ] \in \sH \times \sH \times [L^\infty(Q)]^N$ and a distribution $\mathfrak{z}^\circ \in \sW^*$ such that:
  \begin{subequations} \label{mth4_01}
    \begin{equation}
    (L_u p^\circ + M_u u^\circ, h - u^\circ)_\sH \geq 0, \ \mbox{ for any } h \in \sH, \label{mth4_01a}
  \end{equation}
  \begin{equation}
    L_v z^\circ + M_v v^\circ = 0 \mbox{ in } \sH, \label{mth4_01b}
  \end{equation}
  \end{subequations}
  \begin{gather}
    p^\circ \in W^{1,2}(0,T;V), \, z^\circ \in L^\infty(0,T;V), \label{mth4_02}
    \\
    \varpi^\circ \in \Sgn(\nabla \theta^\circ) \mbox{ a.e. in } Q,
  \end{gather}
  \begin{gather}
    \bigl( \bigl( -\partial_t p^\circ + \bigl( g'(\eta^\circ)+ \alpha''(\eta^\circ) |\nabla \theta^\circ| \bigr) p^\circ + \alpha_0'(\eta^\circ) \xi^\circ + \alpha'(\eta^\circ) \sigma^\circ \bigr)(t), \varphi \bigr)_H 
    \\
    + (\nabla (p^\circ - \mu^2 \partial_t p^\circ)(t), \nabla \varphi)_{[H]^N} = (M_\eta(\eta^\circ - \eta_\ad)(t), \varphi)_H, \label{mth4_03}
    \\
    \mbox{ for any } \varphi \in V, \mbox{ and a.e. } t \in (0,T), \mbox{ with } p^\circ(T) = 0 \mbox{ in } \sH,
  \end{gather}
  and
  \begin{gather}
    (\alpha_0(\eta) z^\circ, \partial_t \psi)_\sH + \langle \mathfrak{z}^\circ, \psi \rangle_\sW + \nu^2(\nabla z^\circ, \nabla \partial_t \psi )_{[\sH]^N} + (\alpha'(\eta^\circ) p^\circ \varpi^\circ, \nabla \psi)_{[\sH]^N} 
    \\
    = (M_\theta (\theta^\circ - \theta_\ad), \psi)_\sH, \ \mbox{ for any } \psi \in \sW. \label{mth4_04}
  \end{gather}
\end{mTh}

\begin{rem}
  If $M_u > 0$ and $L_u > 0$, then \eqref{proj01} enables us to rewrite the variational inequality given in \eqref{NC0a} and \eqref{mth4_01a} as:
  \begin{equation}
    u^* = \proj_\sK\Bigl(- \frac{L_u}{M_u}p_\varepsilon^* \Bigr) \mbox{ in } \sH, \mbox{ and } u^\circ = \proj_\sK\Bigl(-\frac{L_u}{M_u} p^\circ \Bigr) \mbox{ in } \sH.
  \end{equation}
  respectively.
\end{rem}

\section{Key-lemmas}

\subsection{State equation}
In this subsection, we recall some preceding results regarding to pseudo-parabolic KWC systems (S)$_\varepsilon$ (cf. \cite{MR4797677,mizuno2024weak}).

\begin{keylem}[Solvability of (S)$_\varepsilon$, uniqueness and regularity of solution] \label{lem:solvability}
  We assume \\
  (A0)--(A4), and we let $u \in L^\infty(Q)$ and $v \in \sH$. Then, (S)$_\varepsilon$ admits a unique solution $[\eta, \theta] \in [\sH]^2$ in the sense that:
  \begin{equation}
    \eta \in W^{1,2}(0,T;V) \cap L^\infty(Q), \ \theta \in W^{1,2}(0,T;V), \mbox{ and } [\eta(0), \theta(0)] = [\eta_0, \theta_0] \mbox{ in } V,
  \end{equation}
  \begin{gather}
    (\partial_t \eta(t) + g(\eta(t)) + \alpha'(\eta(t)) \gamma_\varepsilon(\nabla \theta(t)), \varphi)_H + (\nabla (\eta + \mu^2 \partial_t \eta)(t), \nabla \varphi)_{[H]^N} 
    \\
    = (L_u u(t), \varphi)_H, \ \mbox{ for any } \varphi \in V, \mbox{ and a.e. } t \in (0,T), \label{(S)1st_VI}
  \end{gather}
  and
  \begin{gather}
    (\alpha_0(\eta(t))\partial_t \theta(t), \theta(t) - \psi)_H + \int_\Omega \alpha(\eta(t)) \gamma_\varepsilon(\nabla \theta(t))\,dx + \nu^2 (\nabla \partial_t \theta(t), \nabla (\theta(t) - \psi))_{[H]^N}
    \\
    \leq \int_\Omega \alpha(\eta(t)) \gamma_\varepsilon(\nabla \psi) \,dx + (L_v v(t), \theta(t) - \psi)_H, \label{(S)2nd_VI}
    \\
    \mbox{ for any } \psi \in V, \ \mbox{ and a.e. } t \in (0,T).
  \end{gather}
\end{keylem}

\begin{rem}
  When $\varepsilon > 0$, standard variational methods allow us to rewrite \eqref{(S)2nd_VI} by:
  \begin{gather}
    (\alpha_0(\eta(t)) \partial_t \theta(t), \psi)_H + (\alpha(\eta(t)) \nabla \gamma_\varepsilon(\nabla \theta(t)) + \nu^2 \nabla \partial_t \theta(t), \nabla \psi)_{[H]^N} = (L_v v(t), \psi)_H,
    \\
    \mbox{for any } \psi \in V, \ \mbox{ and a.e. } t \in (0,T). 
  \end{gather}
\end{rem}

\begin{rem}
  Although a smooth boundary is assumed in the previous work on the pseudo-parabolic KWC system \cite{mizuno2024weak}, the assertions in Key Lemma~\ref{lem:solvability} remain valid under the weaker assumption that the boundary $\Gamma$ is Lipschitz continuous. Furthermore, if $N \in {1,2,3}$, the boundary $\Gamma$ is of class $C^4$, and 
  \begin{equation}
  \{ \eta_0, \theta_0 \} \subset W_0 := \left\{ w \in H^2(\Omega) \,|\, \nabla w \cdot n_\Gamma = 0 \ \mbox{ a.e. on } \Gamma \right\},
  \end{equation}
  then the regularity of the solution $[\eta, \theta]$ improves to
  \begin{equation}
    \eta \in W^{1,2}(0,T; W_0), \quad \theta \in L^\infty(0,T; W_0).
  \end{equation}
\end{rem}

\begin{keylem}[Continuous dependence of solution] \label{lem:CD}
  Let us assume (A0)--(A3). Let 
  \\
  $\{ \varepsilon_n \}_{n=1}^\infty \subset [0,\infty)$, $\{ [\eta_{0,n}, \theta_{0,n}] \}_{n=1}^\infty \subset [V \cap L^\infty(Q)] \times V$ and $\{ [u_n, v_n] \}_{n=1}^\infty \subset L^\infty(Q) \times \sH$ be sequences fulfilling the following conditions:
  \begin{equation}
    \left\{ \begin{aligned}
      &\bullet \ \sup_{n \in \N} |\eta_{0,n}|_{L^\infty(\Omega)} < \infty, \mbox{ and } \sup_{n \in \N} |u_n|_{L^\infty(Q)} < \infty,
      \\
      &\bullet \ \begin{gathered}
        \varepsilon_n \to \varepsilon \mbox{ in } \R, \, \eta_{0,n} \to \eta_0, \, \theta_{0,n} \to \theta_0 \mbox{ in } V, \mbox{ and }
        \\
        \sqrt{L_u} u_n \to \sqrt{L_u} u, \, \sqrt{L_v} v_n \to \sqrt{L_v} v \mbox{ weakly in } \sH, \mbox{ as } n \to \infty.
      \end{gathered}
    \end{aligned} \right.
  \end{equation}
  Let $[\eta, \theta]$ be the unique solution to (S)$_\varepsilon$, corresponding to the initial data $[\eta_0, \theta_0]$ and the forcings $[u,v]$, and let $[\eta_n, \theta_n]$ be the unique solution to (S)$_{\varepsilon_n}$, for the initial data $[\eta_{0,n}, \theta_{0,n}]$ and the forcings $[u_n,v_n]$. Then, we can obtain the following convergences as $n \to \infty$:
  \begin{equation}
    \left\{ \begin{aligned}
      &\eta_n \to \eta \mbox{ in } C([0,T];H), \, L^2(0,T;V), \mbox{ weakly in } W^{1,2}(0,T;V)
      \\
      &\qquad\qquad\mbox{ and weakly-$*$ in } L^\infty(Q),
      \\
      &\theta_n \to \theta \mbox{ in } C([0,T];H), \, L^2(0,T;V) \mbox{ and weakly in } W^{1,2}(0,T;V),
      \\
      &\eta_n(t) \to \eta(t), \, \theta_n(t) \to \theta(t) \mbox{ in } V, \mbox{ for all } t \in [0,T].
    \end{aligned}\right.
  \end{equation}
  In addition, if $N \in \{ 1,2,3 \}$, $\Gamma$ has $C^4$-regularity, and the sequence $\{ [\eta_{0,n}, \theta_{0,n}] \}_{n=1}^\infty$ satisfies that
  \begin{equation}
    \eta_{0,n} \to \eta_0, \, \theta_{0,n} \to \theta_0 \mbox{ weakly in } W_0 \mbox{ as } n \to \infty,
  \end{equation}
  then we can derive stronger convergences such that:
  \begin{equation}
    \left\{ \begin{aligned}
      &\eta_n \to \eta \mbox{ in } C([0,T];V), \mbox{ and weakly in } W^{1,2}(0,T;W_0),
      \\
      &\theta_n \to \theta \mbox{ in } C([0,T];V), \mbox{ and weakly-$*$ in } L^\infty(0,T;W_0),
    \end{aligned} \right. \mbox{ as } n \to \infty.
  \end{equation}
\end{keylem}

\subsection{Pseudo-parabolic linear system associated with (OCP)$_\varepsilon$}
In this section, we consider a pseudo-parabolic linear system, denoted by (P), and given by:
  \begin{align}
      &\mbox{(P)}:
      \\
      & \begin{cases}
          \partial_t p -\mathit{\Delta} \bigl( p +\mu^2 \partial_t p \bigr) + \lambda(t,x) p + \xi(t,x) z + \omega(t,x) \cdot \nabla z = h(t, x), ~ \mbox{for $ (t, x) \in Q $,}
          \\[0.5ex]
          \nabla \bigl( p +\mu^2 \partial_t p \bigr) \cdot n_\Gamma =  0, ~ \mbox{on $ \Sigma $,}
          \
          \\[0.5ex]
          p(0, x) = p_0(x),~ \mbox{for $ x \in \Omega $,}
      \end{cases}
      \label{(P)1st}
      \\[1ex]
      & \begin{cases}
          \ds a(t,x) \partial_t z + b(t,x) z + c(t,x) p -\mathrm{div} \left( A(t,x) \nabla z + \nu^2 \nabla \partial_t z + p \omega(t,x) \right) 
          \\
          \qquad\qquad = k(t, x), ~ \mbox{for $ (t, x) \in Q $,}
          \\[1ex]
          \left( A(t,x) \nabla z + \nu^2 \nabla \partial_t z + p \omega \right) \cdot n_\Gamma = 0 ~ \mbox{on $ \Sigma $,}
          \\[0.5ex]
          z(0, x) = z_0(x),~ \mbox{for $ x \in \Omega $.}
      \end{cases}
      \label{(P)2nd}
  \end{align}
In this context, $[p,z]:Q \longrightarrow [\R^2]$ is the unknown of the system (P). $[p_0, z_0] \in [V]^2$ is the initial data of $[p,z]$, and $[h,k] \in [\sV^*]^2$ is a forcing pair. $[a,b,c,\lambda,\xi,\omega,A] \subset [\sH]^7$ is a given septuplet of functions, which belongs to a subclass $\sS \subset [\sH]^7$, defined as follows:
\begin{equation}
  \sS := \left\{ [\tilde a,\tilde b,\tilde c,\tilde \lambda,\tilde \xi,\tilde \omega,\tilde A] \in [\sH]^7 \left| \begin{aligned}
    &\bullet \tilde a \in W^{1,2}(0,T;H), \mbox{ and } \inf \tilde a(\R) \geq \delta_a > 0,
    \\
    &\bullet \ \tilde \omega \in [L^\infty(Q)]^N,
    \\
    &\bullet \ \tilde A \in [L^\infty(Q)]^{N \times N}, \mbox{ and the value $\tilde A(t,x) \in \R^{N \times N}$ } 
    \\
    &\quad \mbox{ is symmetric and positive, for a.e. $(t,x) \in Q$},
  \end{aligned} \right. \right\}
\end{equation}
where $\delta_a > 0$ is a given constant.

The problem (P) is considered in the following two contexts:
\begin{itemize}
  \item[$\sharp 1$)] linearized system of (S)$_\varepsilon$ associated with the G\^{a}teaux derivative $\sJ_\varepsilon'$ of the cost functional $\sJ_\varepsilon$ at $[u,v] \in \sU_\ad$,
  \item[$\sharp 2$)] the adjoint problem of the linearized system $\sharp 1$). 
\end{itemize}
Let $[\eta,\theta]$ be the solution to (S)$_\varepsilon$. Then, the linearized system $\sharp 1)$ corresponds to the system (P) under:
\begin{gather}
  \left\{ \begin{aligned}
    &a(t,x) = [\alpha_0(\eta)](t,x), \ b(t,x) = 0, \ c(t,x) = [\alpha_0'(\eta) \partial_t \theta](t,x), 
    \\ 
    &\lambda(t,x) = [g'(\eta) + \alpha''(\eta) \gamma_\varepsilon(\nabla \theta)](t,x), \ \xi(t,x) = 0,
    \\
    &\omega(t,x) = [\alpha'(\eta) \nabla \gamma_\varepsilon(\nabla \theta)](t,x), \ A = [\alpha(\eta) \nabla^2 \gamma_\varepsilon(\nabla \theta)](t,x),
    \\
    &[p_0, z_0](x) = [0,0],
  \end{aligned} \right. \label{(P)ex1}
  \\
  \mbox{ for a.e. } (t,x) \in Q. 
\end{gather}
On the other hand, the adjoint problem $\sharp 2)$ coincides with the system (P) for the case when:
\begin{gather}
  \left\{ \begin{aligned}
    &a(t,x) = [\alpha_0(\eta)](T-t,x), \ b(t,x) = [\alpha_0'(\eta) \partial_t \eta](T-t,x), \ c(t,x) = 0,
    \\ 
    &\lambda(t,x) = [g'(\eta) + \alpha''(\eta) \gamma_\varepsilon(\nabla \theta)](T-t,x), \ \xi(t,x) = [\alpha_0'(\eta) \partial_t \theta](T-t,x),
    \\
    &\omega(t,x) = [\alpha'(\eta) \nabla \gamma_\varepsilon(\nabla \theta)](T-t,x), \ A = [\alpha(\eta) \nabla^2 \gamma_\varepsilon(\nabla \theta)](T-t,x),
    \\
    &[p_0, z_0](x) = [0,0],
  \end{aligned} \right. \label{(P)ex2}
  \\
  \mbox{ for a.e. } (t,x) \in Q. 
\end{gather}

Previously, the parabolic type version of (P) is studied in \cite{MR3888633}, which established well-posedness results on the parabolic system using time-discretization method. However, \cite{MR3888633} aims to apply their results to an optimal control problems of (S)$_\varepsilon$ under the setting that $\alpha_0$ does not depend on $\eta$, which is improved in our studies \cite{MR4797677,mizuno2024weak}. Also, the difference between our studies and \cite{MR3888633} leads to additional terms $\alpha_0'(\eta) \partial_t \eta$ and $\alpha_0'(\eta) \partial_t \theta$, which may pose some difficulties in the analysis under standard regularities for parabolic type PDEs. On the other hand, we can expect that the regularizing effect of the pseudo-parabolic structure of (P) enables us to derive some well-posedness results although (P) includes some terms difficult to handle. Based on these, we set our goal of this subsection to establish a well-posedness result, which is useful on the optimal control problems.

\begin{keylem}[Solvability of (P)] \label{lem:Psol}
  For any $[a,b,c,\lambda,\xi,\omega,A] \in \sS$, the pseudo-parabolic system (P) admits a unique solution in the sense that:
  \begin{equation}
    p,z \in W^{1,2}(0,T;V), \mbox{ and } p(0) = p_0, \, z(0) = z_0 \mbox{ in } H,
  \end{equation}
  \begin{gather}
    (\partial_t p(t) + \omega(t) \cdot \nabla z(t), \varphi)_H + \int_\Omega (\lambda(t) p(t) + \xi(t)z(t)) \varphi \,dx +(\nabla p(t), \nabla \varphi)_{[H]^N} \label{P_var1}
    \\
    + \mu^2 (\nabla \partial_t p(t), \nabla \varphi)_{[H]^N} = \langle h(t),\varphi \rangle_V, \mbox{ for any } \varphi \in V, \mbox{ and for a.e. } t \in (0,T),
  \end{gather}
  and
  \begin{gather}
    \int_\Omega (a(t) \partial_t z(t) + b(t) z(t) + c(t) p(t)) \psi \,dx + (A(t) \nabla z(t) + p(t)\omega(t), \nabla \psi)_{[H]^N} \label{P_var2}
    \\
    + \nu^2(\nabla \partial_t z(t), \nabla \psi)_{[H]^N} = \langle k(t),\psi \rangle_V, \ \mbox{ for any }  \psi \in V, \ \mbox{ and for a.e. } t \in (0,T).
  \end{gather}
\end{keylem}

The proof of Key-Lemma \ref{lem:Psol} is based on time-discretization method. Here, we let $\tau \in (0,1)$ be a time-step size, and let $n_\tau := \{n \in \N \,|\, n \tau \geq T \}$. We adopt the following time-discretization scheme (TD)$_\tau$:

\smallskip
(TD)$_\tau$: find a sequence $\{ [p_i, z_i] \}_{i=0}^\infty \subset [V]^2$ such that:
\begin{align}
  &\left\{ \begin{gathered}
    \frac{1}{\tau}(p_i - p_{i-1}, \varphi)_H + \int_\Omega (\lambda_i p_i + \xi_i z_{i-1}) \varphi \,dx + (\omega_i \cdot \nabla z_i, \varphi)_H + (\nabla p_i, \nabla \varphi)_{[H]^N} 
    \\
    + \frac{\mu^2}{\tau} (\nabla (p_i - p_{i-1}), \nabla \varphi)_{[H]^N} = \langle h_i, \varphi\rangle_V, \ \mbox{ for any } \varphi \in V,
  \end{gathered} \right. \label{P_TD1}
  \\
  &\left\{ \begin{gathered}
    \frac{1}{\tau} \int_\Omega a_i \psi (z_i - z_{i-1}) \,dx + \int_\Omega (b_i z_i + c_i p_{i-1}) \psi \,dx + (A_i \nabla z_i + p_i \omega_i, \nabla \psi)_{[H]^N} 
    \\
    + \frac{\nu^2}{\tau} (\nabla(z_i - z_{i-1}), \nabla \psi)_{[H]^N} = \langle k_i, \psi \rangle_V, \ \mbox{ for any } \psi \in V,
  \end{gathered} \right. \label{P_TD2}
  \\
  &\qquad \mbox{ for any } i = 1,2,3,\dots, \mbox{ with $[p_0,z_0]$ is the initial data of $[p,z]$}. 
\end{align}
In this context, the sequences of forcing pairs $\{ [h_i, k_i] \}_{i=1}^\infty \in [V^*]^2$ and septuplets 
\\
$\{ [a_i, b_i, c_i, \lambda_i, \xi_i, \omega_i, A_i] \}_{i=1}^\infty \subset [H]^7$ is given as the time-discretization data of $[h,k]$ and $[a,b,c,\lambda,\xi,\omega,A]$ as in \eqref{tI01} and \eqref{tI05}, respectively. On account of \eqref{tI}, the following convergences hold as $\tau \downarrow 0$:\noeqref{tI02,tI03,tI04,tI05,tI06}
\begin{gather}
  [\overline{h}]_\tau \to h, \ [\overline{k}]_\tau \to k,  \mbox{ in } \sV^*, 
\end{gather}
\begin{gather}
  [\overline{a}]_\tau \to a \mbox{ in } L^\infty(0,T;H),
  \\
  [\overline{b}]_\tau \to b, \ [\overline{c}]_\tau \to c, [\overline{\lambda}]_\tau \to \lambda, [\overline{\xi}]_\tau \to \xi, \mbox{ in } \sH,
  \\
  [\overline{\omega}]_\tau \to \omega \mbox{ weakly-$*$ in } [L^\infty(Q)]^N \mbox{ and in the pointwise sense a.e. in $Q$},
\end{gather}
and
\begin{gather}
  [\overline{A}]_\tau \to A \mbox{ weakly-$*$ in } [L^\infty(Q)]^{N \times N} \mbox{ and in the pointwise sense a.e. in $Q$}.
\end{gather}

\begin{rem} \label{emb}
  Since $N \in \{ 1,2,3,4 \}$ and $\Omega \subset \R^N$ is a bounded domain, $V$ is continuously embedded in $L^4(\Omega)$. Hereafter, $C_V^{L^4}$ denotes the constant of embedding from $V$ to $L^4(\Omega)$. Also, for any $w^\circ \in V$ and any sequence $\{ w_i^\circ \}_{i = 1}^\infty \subset V$ such that 
    \begin{equation}
    w_i^\circ \to w^\circ \mbox{ weakly in } V \mbox{ as } i \to \infty,
  \end{equation}
  one can confirm that
  \begin{equation}
    w_i^\circ \to w^\circ \mbox{ weakly in } L^4(\Omega) \mbox{ as } i \to \infty. 
  \end{equation}
  Besides, for any $\varphi \in L^4(\Omega)$, it is checked that
  \begin{equation}
    w_i^\circ \varphi \to w^\circ \varphi \mbox{ weakly in } H \mbox{ as } i \to \infty. \label{emb01}
  \end{equation}
  Similarly, for any $\varphi \in L^\infty(0,T;L^4(\Omega)), \ w \in \sV$ and for any sequence $\{ w_i \}_{i=1}^\infty \subset \sV$ such that
  \begin{equation}
    w_i \to w \mbox{ weakly in } \sV \mbox{ as } i \to \infty,
  \end{equation} 
  it can be checked that
  \begin{equation}
    w_i \varphi \to w \varphi \mbox{ weakly in } \sH \mbox{ as } i \to \infty. \label{emb02}
  \end{equation}
  
\end{rem}

To find the solution to (TD)$_\tau$, we prepare a fundamental lemma.

\begin{keylem} \label{lem:P_TD}
  We let $[a^\circ,b^\circ,c^\circ,\lambda^\circ,\xi^\circ,\omega^\circ,A^\circ] \in [H]^7$ be a fixed septuplet of functions satisfying that 
  \begin{equation}
    \left\{ \begin{aligned}
      &|b^\circ|_H \leq \tau^{-\frac{1}{2}}|b|_\sH, \ |\lambda^\circ|_H \leq \tau^{-\frac{1}{2}}|\lambda|_\sH, \ |\omega^\circ|_{[L^\infty(\Omega)]^N} \leq |\omega|_{[L^\infty(Q)]^N}, \ \inf a^\circ(\R) \geq \delta_a > 0, 
      \\
      &|A^\circ|_{[L^\infty(\Omega)]^{N \times N}} \leq |A|_{[L^\infty(Q)]^{N \times N}}, \mbox{ and $A^\circ$ is positive and symmetric, a.e. on $\Omega$}. 
    \end{aligned} \right. \label{P_exist0}
  \end{equation}
  Also, we assume that:  
  \begin{align}
    0 < \tau < \tau_1 &= \tau_1(a,b,\lambda,\omega) := \frac{1 \wedge \delta_a \wedge \mu^2 \wedge \nu^2}{8 \bigl( (C_V^{L^4})^2 + 1 \bigr)\bigl( |\lambda|_\sH + |b|_\sH + |\omega|_{L^\infty(Q)} + 1\bigr)}, \label{P_exist1}
  \end{align}
  so that:
  \begin{equation}
    \left\{ \begin{aligned}
      &\tau \cdot \frac{\bigl( (C_V^{L^4})^2 + 1 \bigr)}{1 \wedge \mu^2} \Bigl( |\lambda^\circ|_H + |\omega^\circ|_{[L^\infty(\Omega)]^N} + 1 \Bigr) < \frac{1}{8},
      \\
      &\tau \cdot \frac{\bigl( (C_V^{L^4})^2 + 1 \bigr)}{1 \wedge \delta_a \wedge \nu^2} \Bigl( |b^\circ|_H + |\omega^\circ|_{[L^\infty(\Omega)]^N} + 1 \Bigr) < \frac{1}{8},
    \end{aligned} \right. \ \mbox{ for every } \tau \in (0,\tau_1). \label{P_exist2}
  \end{equation}
  Then, for any $[p^\dagger, z^\dagger] \in [V]^2$ and $[h^\circ, k^\circ] \in [V^*]^2$, the following variational system admits a unique solution $[p,z] \in [V]^2$:
  \begin{align}
    &\left\{ \begin{gathered}
      \frac{1}{\tau}(p - p^\dagger, \varphi)_H + \int_\Omega (\lambda^\circ p + \xi^\circ z^\dagger) \varphi \,dx + (\omega^\circ \cdot \nabla z, \varphi)_H + (\nabla p, \nabla \varphi)_{[H]^N} 
      \\
      + \frac{\mu^2}{\tau} (\nabla (p - p^\dagger), \nabla \varphi)_{[H]^N} = \langle h^\circ, \varphi \rangle_V, \ \mbox{ for any } \varphi \in V,
    \end{gathered} \right. \label{P_TD3}
    \\
    &\left\{ \begin{gathered}
      \frac{1}{\tau} \int_\Omega a^\circ \psi (z - z^\dagger) \,dx + \int_\Omega (b^\circ z + c^\circ p^\dagger) \psi \,dx + (A^\circ \nabla z + p \omega^\circ, \nabla \psi)_{[H]^N} 
      \\
      + \frac{\nu^2}{\tau} (\nabla(z - z^\dagger), \nabla \psi)_{[H]^N} = \langle k^\circ, \psi \rangle_V, \ \mbox{ for any } \psi \in V.
    \end{gathered} \right. \label{P_TD4}
  \end{align}
\end{keylem}

\begin{proof}
  We define a (non-convex) functional $\Upsilon: [H]^2 \longrightarrow (-\infty,\infty]$, by:
  \begin{align}
    &\Upsilon: [p,z] \in [H]^2 \mapsto \Upsilon(p,z) 
    \\
    &\quad := \left\{ \begin{aligned}
      &\frac{1}{2\tau} \bigl(|p - p^\dagger|_H^2 + |\sqrt{a^\circ}(z - z^\dagger)|_H^2\bigr) + \frac{1}{2}|\nabla p|_{[H]^N}^2
      \\
      &\quad + \frac{1}{2\tau} \bigl(\mu^2|\nabla(p - p^\dagger)|_{[H]^N}^2 + \nu^2|\nabla (z - z^\dagger)|_{[H]^N}^2\bigr)
      \\
      &\quad + \int_\Omega \bigl(\lambda^\circ |p|^2 + b^\circ |z|^2 + \xi^\circ z^\dagger p + c^\circ p^\dagger z + (\omega^\circ \cdot \nabla z) p + |A^{\frac{1}{2}} \nabla z|^2\bigr) \,dx
      \\
      &\quad - \langle h^\circ, p \rangle_V - \langle k^\circ,z \rangle_V, \mbox{ if } [p,z] \in V,
      \\[1ex]
      &+\infty, \ \mbox{otherwise}.
    \end{aligned} \right.
  \end{align}
  On account of \eqref{P_exist1} and \eqref{P_exist2}, we can check that $\Upsilon$ is proper and weakly l.s.c. on $[H]^2$, via the following lower-bound estimate of $\Upsilon$:
  \begin{align} 
    \Upsilon(p,z) &\geq \frac{1 \wedge \delta_a \wedge \mu^2 \wedge \nu^2}{8\tau} (|p|_V^2 + |z|_V^2) - \left( \frac{1 \wedge \mu^2}{2\tau} + \frac{(C_V^{L^4})^2}{2}|c^\circ|_H^2 \right)|p^\dagger|_V^2 
    \\
    &\quad - \left(\frac{\delta_a \wedge \nu^2}{2\tau} + \frac{(C_V^{L^4})^2}{2}|\xi^\circ|_H^2 \right) |z^\dagger|_V^2 - \frac{1}{2}(|h^\circ|_H^2 + |k^\circ|_H^2),
    \\
    &\qquad\qquad\qquad\qquad \mbox{ for any } [p,z] \in [V]^2,
  \end{align}
  which is obtained by the following computations:
  \begin{align}
    &\frac{1}{2\tau} \bigl(|p - p^\dagger|_H^2 + |\sqrt{a^\circ}(z - z^\dagger)|_H^2\bigr) + \frac{1}{2\tau} \bigl(\mu^2|\nabla(p - p^\dagger)|_{[H]^N}^2 + \nu^2|\nabla (z - z^\dagger)|_{[H]^N}^2\bigr)
    \\
    &\quad \geq \frac{1}{2 \tau}\bigl( (1 \wedge \mu^2)|p - p^\dagger|_V^2 + (\delta_a \wedge \nu^2) |z - z^\dagger|_V^2 \bigr)
    \\
    &\quad \geq \frac{1}{4 \tau}\bigl( (1 \wedge \mu^2)|p|_V^2 + (\delta_a \wedge \nu^2) |z|_V^2 \bigr) - \frac{1}{2 \tau}\bigl( (1 \wedge \mu^2)|p^\dagger|_V^2 + (\delta_a \wedge \nu^2) |z^\dagger|_V^2 \bigr),
  \end{align}
  \begin{align}
    &\quad \int_\Omega (\lambda^\circ |p|^2 + b^\circ |z|^2) \,dx + \int_\Omega (\omega^\circ \cdot \nabla z) p \,dx
    \\
    &\geq -(C_V^{L^4})^2 |\lambda^\circ|_H |p|_V^2 - (C_V^{L^4})^2 |b^\circ|_H |z|_V^2 - |\omega^\circ|_{[L^\infty(\Omega)]^N} |\nabla z|_{[H]^N} |p|_H
    \\
    &\geq -((C_V^{L^4})^2 + 1) \Bigl((|\lambda^\circ|_H + |\omega^\circ|_{[L^\infty(\Omega)]^N}) |p|_V^2 + (|b^\circ|_H + |\omega^\circ|_{[L^\infty(\Omega)]^N}) |z|_V^2 \Bigr),
  \end{align}
  \begin{align}
    \int_\Omega (\xi^\circ z^\dagger p + c^\circ p^\dagger z) \,dx \geq -\frac{(C_V^{L^4})^2}{2} \bigl( |\xi^\circ|_H^2 |z^\dagger|_V^2 + |p|_V^2 + |c^\circ|_H^2 |p^\dagger|_V^2 + |z|_V^2 \bigr),
  \end{align}
  and 
  \begin{equation}
    -\dual{h^\circ,p}_V - \dual{k^\circ,z}_V \geq -\frac{1}{2}(|h^\circ|_{V^*}^2 + |p|_V^2) - \frac{1}{2}(|k^\circ|_{V^*}^2 + |z|_V^2).
  \end{equation}
  Therefore, we can find a minimizer of $\Upsilon$, and see the minimizer solves the variational system \eqref{P_TD3}--\eqref{P_TD4} by means of the direct method of calculus of variations.

  Next, for the proof of uniqueness, let $[p^j, z^j] \in V \times V$ be the solution to the system \eqref{P_TD3}--\eqref{P_TD4} for $j = 1,2$. By taking the difference between the systems \eqref{P_TD3}--\eqref{P_TD4} for $j=1,2$, and putting $[\varphi,\psi] = [p^1 - p^2, z^1 - z^2]$, it is observed that:
  \begin{align}
    &\frac{1 \wedge \mu^2}{\tau}|p^1 - p^2|_V^2 + \frac{\delta_a \wedge \nu^2}{\tau}|z^1 - z^2|_V^2 - (C_V^{L^4})^2 |\lambda^\circ|_H |p^1 - p^2|_V^2 
    \\
    &\quad - (C_V^{L^4})^2 |b^\circ|_H |z^1 - z^2|_V^2 - |\omega^\circ|_{[L^\infty(\Omega)]^N} (|\nabla (z^1 - z^2)|_{[H]^N}^2 + |p^1 - p^2|_H^2) \leq 0.
  \end{align}
  Having in mind \eqref{P_exist2}, we obtain that
  \begin{equation}
    \frac{1 \wedge \mu^2}{2\tau}|p^1 - p^2|_V^2 + \frac{\delta_a \wedge \nu^2}{2\tau}|z^1 - z^2|_V^2 \leq 0, \mbox{ for all } \tau \in (0,\tau_1), \label{P_TD5}
  \end{equation}
  and \eqref{P_TD5} implies that the solution to the system \eqref{P_TD3}--\eqref{P_TD4} is unique.
\end{proof}

For the proof of Key-Lemma \ref{lem:Psol}, the following lemma will be pivotal.

\begin{keylem} \label{lem:P_est}
  For any $\tau \in (0,\tau_1)$, our time-discretization scheme (TD)$_\tau$ admits a unique solution. Moreover, There exists a small time-step size $\tau_2 \in (0,\tau_1)$ such that the following estimates (i) and (ii) hold.
  
  \smallskip
  \noindent
  (i) Let $C_1 = C_1(a,b,c,\lambda,\xi,\omega)$ be a positive constant which depends on $|\partial_t a|_\sH$, $|b|_\sH$, $|c|_\sH$, $|\lambda|_\sH$, $|\xi|_\sH$, and $|\omega|_{[L^\infty(Q)]^N}$, given as:
  \begin{gather}
    C_1 := \frac{8((C_V^{L^4})^2 + 1)}{1 \wedge \delta_a \wedge \mu^2 \wedge \nu^2}\bigl( |\partial_t a|_\sH + |b|_\sH + |c|_\sH + |\lambda|_\sH + |\xi|_\sH + |\omega|_{[L^\infty(Q)]^N} + 1\bigr).
  \end{gather}
  Then, it holds that:
  \begin{align}
    &|p_i|_H^2 + \mu^2 |\nabla p_i|_{[H]^N}^2 + |\sqrt{a_i}z_i|_H^2 + \nu^2 |\nabla z_i|_{[H]^N}^2 
    \\
    &\quad \leq e^{4 C_1 (T + 1)} \bigl( |p_0|_H^2 + \mu^2 |\nabla p_0|_{[H]^N}^2 + |\sqrt{a(0)}z_0|_H^2 + \nu^2 |\nabla z_0|_{[H]^N}^2 \bigr)
    \\
    &\quad\qquad + 2(|h|_{\sV^*}^2 + |k|_{\sV^*}^2), \mbox{ for any } i = 1,2,3,\dots, n_\tau. \label{P_TD_est1}
  \end{align}
  (ii) There exists a positive constant $C_2 = C_2(a) > 0$, depending on $\delta_a$ and $|a(0)|_H$ such that
  \begin{gather}
    \frac{1 \wedge \mu^2}{4\tau}\sum_{i=1}^{n_\tau} |p_i - p_{i-1}|_V^2 \leq C_2 e^{5C_1 (T + 1)} (|p_0|_V^2 + |z_0|_V^2 + |h|_{\sV^*}^2 + |k|_{\sV^*}^2), \label{P_TD_est2}
  \end{gather}
  and
  \begin{align}
    &\frac{\delta_a \wedge \nu^2}{4\tau}\sum_{i=1}^{n_\tau} |z_i - z_{i-1}|_V^2 
    \\
    &\qquad \leq C_2 e^{5 C_1 (T + 1)}(|A|_{[L^\infty(Q)]^{N \times N}}^2 + 1) (|p_0|_V^2 + |z_0|_V^2 + |h|_{\sV^*}^2 + |k|_{\sV^*}^2). \label{P_TD_est3}
  \end{align}
\end{keylem}
\begin{proof}
  We note that the time-discretization data $[a_i, b_i, c_i, \lambda_i, \xi_i, \omega_i, A_i]$ satisfy the conditions \eqref{P_exist0}--\eqref{P_exist2}. Hence, by applying Key-Lemma \ref{lem:P_TD} with
  \begin{equation}
    [a^\circ, b^\circ, c^\circ, \lambda^\circ, \xi^\circ, \omega^\circ, A^\circ] = [a_i, b_i, c_i, \lambda_i, \xi_i, \omega_i, A_i] \mbox{ in } [H]^7,
  \end{equation}
  \begin{equation}
    [p^\dagger, z^\dagger] = [p_{i-1}, z_{i-1}] \mbox{ in } [H]^2,
  \end{equation}
  and
  \begin{equation}
    [h^\circ, k^\circ] = [h_i, k_i] \in [V^*]^2,
  \end{equation}
  unique solution $\{[p_i, z_i]\}_{i=1}^\infty$ to (TD)$_\tau$ can be obtained for every $\tau \in (0,\tau_1)$.

  Next, to derive the estimate \eqref{P_TD_est1}, we multiply the both side of \eqref{P_TD1} by $p_i$. Then, by using H\"older's and Young's inequalities and continuous embedding from $V$ to $L^4(\Omega)$, we have:
  \begin{gather}
    \frac{1}{2\tau}(|p_i|_H^2 - |p_{i-1}|_H^2) + \frac{\mu^2}{2\tau} (|\nabla p_i|_{[H]^N}^2 - |\nabla p_{i-1}|_{[H]^N}^2) - (C_V^{L^4})^2 |\lambda_i|_H |p_i|_V^2 \label{P_TD_est4}
    \\
    - \frac{(C_V^{L^4})^2}{2}|\xi_i|_H (|p_i|_V^2 + |z_{i-1}|_V^2) - \frac{|\omega_i|_{L^\infty(\Omega)}}{2}(|p_i|_H^2 + |\nabla z_i|_{[H]^N}^2) \leq \frac{1}{2}(|h_i|_{V^*}^2 + |p_i|_V^2),
    \\
    \mbox{for any } i = 1,2,3,\dots, 
  \end{gather}
  via the following computations:
  \begin{gather}
    \frac{1}{\tau}(p_i - p_{i-1}, p_i)_H + |\nabla p_i|_{[H]^N}^2 + \frac{\mu^2}{\tau}(\nabla (p_i - p_{i-1}), \nabla p_i)_{[H]^N}
    \\
    \geq \frac{1}{2\tau}\bigl( |p_i|_H^2 - |p_{i-1}|_H^2 \bigr) + \frac{\mu^2}{2\tau} \bigl( |\nabla p_i|_{[H]^N}^2 - |\nabla p_{i-1}|_{[H]^N}^2 \bigr),
  \end{gather}
  \begin{align}
      &\quad \int_\Omega (\lambda_i p_i + \xi_i z_{i-1}) p_i \,dx + (\omega_i \cdot \nabla z_i, p_i)_H
      \\
      &\geq - |\lambda_i|_H |p_i|_{L^4(\Omega)}^2 - |\xi|_H |p_i|_{L^4(\Omega)} |z_{i-1}|_{L^4(\Omega)} - |\omega_i|_{L^\infty(\Omega)}|\nabla z_i|_{[H]}|p_i|_H
      \\
      &\geq - (C_V^{L^4})^2 |\lambda_i|_H |p_i|_V^2 - \frac{(C_V^{L^4})^2}{2}|\xi_i|_H (|p_i|_V^2 + |z_{i-1}|_V^2) - \frac{|\omega_i|_{L^\infty(\Omega)}}{2}(|p_i|_H^2 + |\nabla z_i|_{[H]^N}^2)
  \end{align}
  and
  \begin{gather}
    \dual{h_i, p_i}_V \leq \frac{1}{2}(|h_i|_{V^*}^2 + |p_i|_V^2).
  \end{gather}
  In addition, substituting $\psi$ in \eqref{P_TD2} by $z_i$, and noting that $A_i(x)$ is positive for a.e.  $x \in \Omega$, we can compute as follows:
  \begin{gather}
    \frac{1}{\tau}\int_\Omega a_i z_i (z_i - z_{i-1}) \,dx + \frac{\nu^2}{2\tau} (|\nabla z_i|_{[H]^N}^2 - |\nabla z_{i-1}|_{[H]^N}^2) - (C_V^{L^4})^2 |b_i|_H |z_i|_V^2 \label{P_TD_est5}
    \\
    - \frac{(C_V^{L^4})^2}{2}|c_i|_H(|p_{i-1}|_V^2 + |z_i|_V^2) + \frac{|\omega_i|_{L^\infty(\Omega)}}{2}(|p_i|_H^2 + |\nabla z_i|_H^2) \leq \frac{1}{2}(|k_i|_{V^*}^2 + |z_i|_V^2),
    \\
    \mbox{for every } i = 1,2,3,\dots,
  \end{gather}
  via the following calculations:
  \begin{equation}
    \frac{\nu^2}{\tau}(\nabla(z_i - z_{i-1}), \nabla z_i)_{[H]^N} \geq \frac{\nu^2}{2\tau}(|\nabla z_i|_{[H]^N}^2 - |\nabla z_{i-1}|_{[H]^N}^2),
  \end{equation}
  \begin{align}
    &\quad \int_\Omega (b_i z_i + c_i p_{i-1}) z_i \,dx + (A_i \nabla z_i + \omega_i p_i, \nabla z_i)_{[H]^N}
    \\
    &\geq - |b_i|_H |z_i|_{L^4(\Omega)}^2 - |c_i|_H |p_{i-1}|_{L^4(\Omega)} |z_i|_{L^4(\Omega)} - |\omega_i|_{L^\infty(\Omega)} |p_i|_H |\nabla z_i|_{[H]^N}^2
    \\
    &\geq - (C_V^{L^4})^2 |b_i|_H |z_i|_V^2 - \frac{(C_V^{L^4})^2}{2}|c_i|_H(|p_{i-1}|_V^2 + |z_i|_V^2) + \frac{|\omega_i|_{L^\infty(\Omega)}}{2}(|p_i|_H^2 + |\nabla z_i|_H^2),
  \end{align}
  and
  \begin{equation}
    \dual{k_i, z_i}_V \leq \frac{1}{2}(|k_i|_{V^*}^2 + |z_i|_V^2).
  \end{equation}
  By H\"older's and Young's inequalities and the continuous embedding from $V$ to $L^4(\Omega)$, the first term of \eqref{P_TD_est5} can be computed by:
  \begin{align}
    &\quad\frac{1}{\tau} \int_\Omega a_i z_i(z_i - z_{i-1}) \,dx
    \\
    &= \frac{1}{\tau}|\sqrt{a_i} z_i|_H^2 - \frac{1}{\tau} \int_\Omega a_i z_i z_{i-1} \,dx
    \\
    &\geq \frac{1}{2\tau} \bigl( |\sqrt{a_i} z_i|_H^2 - |\sqrt{a_{i-1}} z_{i-1}|_H^2 \bigr) - \frac{1}{2\tau} \int_\Omega (a_i - a_{i-1}) |z_{i-1}|^2 \,dx \label{P_TD_est6}
    \\
    &\geq \frac{1}{2\tau} \bigl( |\sqrt{a_i} z_i|_H^2 - |\sqrt{a_{i-1}} z_{i-1}|_H^2 \bigr) - \frac{1}{2} \left| \frac{a_i - a_{i-1}}{\tau} \right|_H |z_{i-1}|_{L^4(\Omega)}^2
    \\
    &\geq \frac{1}{2\tau} \bigl( |\sqrt{a_i} z_i|_H^2 - |\sqrt{a_{i-1}} z_{i-1}|_H^2 \bigr) - \frac{(C_V^{L^4})^2}{2} \left| \frac{a_i - a_{i-1}}{\tau} \right|_H |z_{i-1}|_V^2.
  \end{align}
  According to \eqref{P_TD_est6}, \eqref{P_TD_est5} can be reduced to
  \begin{align}
    &\quad \frac{1}{2\tau}\bigl( |\sqrt{a_i} z_i|_H^2 - |\sqrt{a_{i-1}} z_{i-1}|_H^2 \bigr) + \frac{\nu^2}{2\tau} (|\nabla z_i|_{[H]^N}^2 - |\nabla z_{i-1}|_{[H]^N}^2) - (C_V^{L^4})^2 |b_i|_H |z_i|_V^2 
    \\
    &\quad - \frac{(C_V^{L^4})^2}{2}|c_i|_H(|p_{i-1}|_V^2 + |z_i|_V^2) + \frac{|\omega_i|_{L^\infty(\Omega)}}{2}(|p_i|_H^2 + |\nabla z_i|_H^2) \label{P_TD_est7}
    \\
    &\quad - \frac{(C_V^{L^4})^2}{2} \left| \frac{a_i - a_{i-1}}{\tau} \right|_H |z_{i-1}|_V^2 \leq \frac{1}{2}(|k_i|_{V^*}^2 + |z_i|_V^2), \ \mbox{ for any } i = 1,2,3,\dots.
  \end{align}

  Now, from \eqref{P_TD_est4} and \eqref{P_TD_est7}, we arrive at the following inequality:
  \begin{equation}
    \frac{1}{\tau}(X_i - X_{i-1}) \leq F_i (X_i + X_{i-1}) + |h_i|_{V^*}^2 + |k_i|_{V^*}^2, \ \mbox{ for } \tau \in (0,\tau_1), \mbox{ and } i = 1,2,3,\dots, \label{P_TD_est8}
  \end{equation}
  where
  \begin{align}
    &X_i := |p_i|_H^2 + \mu^2 |\nabla p_i|_{[H]^N}^2 + |\sqrt{a_i} z_i|_H^2 + \nu^2 |\nabla z_i|_{[H]^N}^2, \mbox{ for each } i = 0,1,2,\dots,
  \end{align}
  and
  \begin{gather}
    F_i := \frac{2((C_V^{L^4})^2 + 1)}{1 \wedge \delta_a \wedge \mu^2 \wedge \nu^2} \left( \left| \frac{a_i - a_{i-1}}{\tau} \right|_H + |b_i|_H + |c_i|_H + |\lambda_i|_H + |\xi_i|_H + |\omega_i|_{L^\infty(\Omega)} + 1 \right),
    \\
    \mbox{ for every } i = 1,2,3,\dots.
  \end{gather}
  We can easily check that
  \begin{align}
    \tau \sum_{j=1}^{n_\tau} F_i \leq C_1 (T + 1), \mbox{ for any } \tau \in (0,\tau_1),
  \end{align}
  Here, we set $\tau_2 \in (0,\tau_1)$ to fulfill:
  \begin{equation}
    \tau_2 := \frac{1 \wedge \delta_a^2 \wedge \mu^4 \wedge \nu^4}{16 ((C_V^{L^4})^2 + 1)^2 \bigl( |\lambda|_\sH + |\xi|_\sH + |b|_\sH + |c|_\sH + |\partial_t a|_\sH + |\omega|_{L^\infty(Q)}+ 1\bigr)^2} \, \wedge \tau_1.
  \end{equation}
  Then, by noting the following inequalities which hold when $\tau \in (0,1)$:
  \begin{equation}
    \tau \bigl( |\lambda_i|_H + |\xi_i|_H + |b_i|_H + |c_i|_H \bigr) \leq \tau^{\frac{1}{2}} (|\lambda|_\sH + |\xi|_\sH + |b|_\sH + |c|_\sH),
  \end{equation}
  \begin{equation}
    \tau|\omega_i|_{L^\infty(\Omega)} \leq \tau |\omega|_{L^\infty(Q)} \leq \tau^{\frac{1}{2}}|\omega|_{L^\infty(Q)}
  \end{equation}
  and 
  \begin{align}
    \tau \left| \frac{a_i - a_{i-1}}{\tau} \right|_H &\leq \left| \int_{t_{i-1}}^{t_i} \partial_t a(t) \,dt \right|_H \leq \int_{t_{i-1}}^{t_i} |\partial_t a(t)|_H \,dt \leq \tau^{\frac{1}{2}} |\partial_t a|_\sH,
  \end{align}
  it holds that:
  \begin{align}
    \tau F_i &\leq \tau^{\frac{1}{2}} \cdot \frac{2((C_V^{L^4})^2 + 1)}{1 \wedge \delta_a \wedge \mu^2 \wedge \nu^2} \left( |\lambda|_\sH + |\xi|_\sH + |b|_\sH + |c|_\sH + \left| \partial_t a \right|_\sH + |\omega|_{L^\infty(Q)} + 1 \right),
    \\
    &\leq \frac{1}{2}, \quad \mbox{ for any } \tau \in (0,\tau_2). \label{P_TD_est9}
  \end{align}
  Hence, by applying the discrete version of Gronwall's lemma (cf. \cite[Section 3.1]{emmrich1999discrete}) to \eqref{P_TD_est8} and having in mind \eqref{P_TD_est9}, it is observed that:
  \begin{align}
    X_i &\leq \exp\left( \sum_{j=1}^{n_\tau} \frac{2\tau F_j}{1 - \tau F_j} \right) X_0  + \sum_{j=1}^{n_\tau} \frac{\tau (|h_j|_{V^*}^2 + |k_j|_{V^*}^2)}{1 - \tau F_j}
    \\
    &\leq \exp\left( 4 \sum_{j=1}^{n_\tau} \tau F_j \right) X_0 + 2 (|h|_{\sV^*}^2 + |k|_{\sV^*}^2)
    \\
    &\leq e^{4 C_1 (T + 1)} X_0 + 2 (|h|_{\sV^*}^2 + |k|_{\sV^*}^2), \ \mbox{ for } i = 1,2,3, \dots, n_\tau.
  \end{align}
  Thus, we obtain the estimate \eqref{P_TD_est1}, and we immediately obtain the following estimate:
  \begin{gather}
    |p_i|_V^2 + |z_i|_V^2 \leq \frac{1 + \mu^2 + \nu^2 + (C_V^{L^4})^2|a(0)|_H}{1 \wedge \delta_a \wedge \mu^2 \wedge \nu^2}e^{4 C_1 (T+1)}(|p_0|_V^2 + |z_0|_V^2 + |h|_{\sV^*}^2 + |k|_{\sV^*}^2),
    \\
    \mbox{ for every } i = 1,2,3,\dots, n_\tau.
  \end{gather}

  Next, to derive the estimate \eqref{P_TD_est2}, we let $\varphi = p_i - p_{i-1}$ and obtain that
  \begin{align}
    &\frac{1 \wedge \mu^2}{4\tau} |p_i - p_{i-1}|_V^2 + \frac{1}{2}(|\nabla p_i|_{[H]^N}^2 - |\nabla p_{i-1}|_{[H]^N}^2) 
    \\
    &\qquad- \frac{2(C_V^{L^4})^4 \tau}{1 \wedge \mu^2}(|\lambda_i|_H^2 + |\xi_i|_H^2)(|p_i|_V^2 + |z_{i-1}|_V^2) - \tau |\omega_i|_{L^\infty(\Omega)}^2 |\nabla z_i|_{[H]^N}^2
    \\
    &\qquad \leq \frac{\tau}{1 \wedge \mu^2} |h_i|_{V^*}^2, \quad \mbox{ for any } i = 1,2,3,\dots, \label{P_TD_est10}
  \end{align}
  via the following computations:
  \begin{gather}
    (\nabla p_i, \nabla(p_i - p_{i-1}))_{[H]^N} \geq \frac{1}{2}(|\nabla p_i|_{[H]^N}^2 - |\nabla p_{i-1}|_{[H]^N}^2),
  \end{gather}
  \begin{align}
    &\int_\Omega (\lambda_i p_i + \xi_i z_{i-1}) (p_i - p_{i-1}) \,dx \geq - (|\lambda_i|_H |p_i|_{L^4(\Omega)} + |\xi_i|_H |z_{i-1}|_{L^4(\Omega)})|p_i - p_{i-1}|_{L^4(\Omega)}
    \\
    &\geq - (C_V^{L^4})^2 (|\lambda_i|_H |p_i|_V + |\xi_i|_H |z_{i-1}|_V)|p_i - p_{i-1}|_V
    \\
    &\geq -\frac{1 \wedge \mu^2}{4\tau}|p_i - p_{i-1}|_V^2 - \frac{2(C_V^{L^4})^4 \tau}{1 \wedge \mu^2} (|\lambda_i|_H^2 |p_i|_V^2 + |\xi_i|_H^2 |z_{i-1}|_V^2),
  \end{align}
  \begin{gather}
    (\omega_i \cdot \nabla z_i, p_i - p_{i-1})_H \geq -\frac{1}{4\tau}|p_i - p_{i-1}|_H^2 - \tau |\omega_i|_{L^\infty(\Omega)}^2 |\nabla z_i|_{[H]^N}^2,
  \end{gather}
  and
  \begin{gather}
    \dual{h_i, p_i - p_{i-1}}_V \leq \frac{1 \wedge \mu^2}{4\tau}|p_i - p_{i-1}|_V^2 + \frac{\tau}{1 \wedge \mu^2} |h_i|_{V^*}^2.
  \end{gather}
  By summing up the both sides of \eqref{P_TD_est10} for $i = 1,2,3,\dots, n_\tau$ and having \eqref{P_TD_est1} in mind, \eqref{P_TD_est2} is confirmed with the following calculations:
  \begin{align}
    &\quad \frac{1 \wedge \mu^2}{4\tau} \sum_{i = 1}^{n_\tau} |p_i - p_{i-1}|_V^2
    \\
    &\leq \frac{1}{2}|\nabla p_0|_{[H]^N}^2 + \frac{2((C_V^{L^4})^4 + 1)}{1 \wedge \mu^2} \cdot \tau\sum_{i=1}^{n_\tau}(|\lambda_i|_H^2 + |\xi_i|_H^2 + |\omega_i|_{[L^\infty(\Omega)]^N}^2) \cdot 
    \\
    &\qquad \cdot (|p_i|_V^2 + |z_i|_V^2 + |z_{i-1}|_V^2) + \frac{\tau}{1 \wedge \mu^2}\sum_{i=1}^{n_\tau} |h_i|_{V^*}^2
  \end{align}
  \begin{align}
    &\leq \frac{1}{2}|\nabla p_0|_{[H]^N}^2 + \frac{4((C_V^{L^4})^4 + 1)(T+1)}{1 \wedge \mu^2} (|\lambda|_\sH^2 + |\xi|_\sH^2 + |\omega|_{[L^\infty(Q)]^N}^2) \cdot 
    \\
    &\quad \cdot\frac{1 + \mu^2 + \nu^2 + (C_V^{L^4})^2|a(0)|_H}{1 \wedge \delta_a \wedge \mu^2 \wedge \nu^2} e^{4 C_1 (T + 1)}(|p_0|_V^2 + |z_0|_V^2 + |h|_{\sV^*}^2 + |k|_{\sV^*}^2) + \frac{1}{1 \wedge \mu^2}|h|_{\sV^*}^2
    \\
    &\leq \frac{9((C_V^{L^4})^4 + 1)(T+1)(1 + \mu^2 + \nu^2 + (C_V^{L^4})^2|a(0)|_H)}{1 \wedge \delta_a^2 \wedge \mu^4 \wedge \nu^4} \cdot
    \\
    &\qquad\qquad \cdot e^{5C_1 (T+1)}(|p_0|_V^2 + |z_0|_V^2 + |h|_{\sV^*}^2 + |k|_{\sV^*}^2).
  \end{align}
  
  Finally, to obtain the estimate \eqref{P_TD_est3}, we multiply the both side of \eqref{P_TD2} by $z_i - z_{i-1}$ and we have
  \begin{gather}
    \frac{\delta_a \wedge \nu^2}{4\tau}|z_i - z_{i-1}|_V^2 - \frac{2(C_V^{L^4})^4 \tau}{\delta_a \wedge \nu^2}(|b_i|_H^2 |z_i|_V^2+ |c_i|_H^2 |p_{i-1}|_V^2)
    \\
    -\frac{2\tau}{\nu^2}(|A_i|_{[L^\infty(\Omega)]^{N \times N}}^2 |\nabla z_i|_H^2 + |\omega_i|_{[L^\infty(\Omega)]^N}^2 |p_i|_H^2) \leq \frac{\tau}{\delta_a \wedge \nu^2} |k_i|_{V^*}^2, \label{P_TD_est11}
    \\
    \mbox{ for any } i = 1,2,3,\dots,
  \end{gather}
  via the following computations:
  \begin{align}
    &\int_\Omega (b_i z_i + c_i p_{i-1})(z_i - z_{i-1})\,dx \geq - (|b_i|_H |z_i|_{L^4(\Omega)} + |c_i|_H |p_{i-1}|_{L^4(\Omega)}) |z_i - z_{i-1}|_{L^4(\Omega)}
    \\
    &\quad \geq -(C_V^{L^4})^2 (|b_i|_H |z_i|_V + |c_i|_H |p_{i-1}|_V) |z_i - z_{i-1}|_V
    \\
    &\quad \geq -\frac{\delta_a \wedge \nu^2}{4\tau}|z_i - z_{i-1}|_V^2 - \frac{2(C_V^{L^4})^4 \tau}{\delta_a \wedge \nu^2}(|b_i|_H^2 |z_i|_V^2 + |c_i|_H^2 |p_{i-1}|_V^2),
  \end{align}
  \begin{align}
    &\quad(A_i \nabla z_i + p_i \omega_i, z_i - z_{i-1})_{[H]^N} 
    \\
    &\geq -(|A_i|_{[L^\infty(\Omega)]^{N \times N}}|\nabla z_i|_H + |\omega_i|_{[L^\infty(\Omega)]^N} |p_i|_H)|\nabla (z_i - z_{i-1})|_{[H]^N}
    \\
    &\geq -\frac{\nu^2}{4\tau} |\nabla (z_i - z_{i-1})|_{[H]^N}^2 - \frac{2\tau}{\nu^2} (|A_i|_{[L^\infty(\Omega)]^{N \times N}}^2 |\nabla z_i|_{[H]^N}^2 + |\omega_i|_{[L^\infty(\Omega)]^N}^2 |p_i|_H^2),
  \end{align}
  and
  \begin{gather}
    \dual{k_i, z_i - z_{i-1}}_V \leq \frac{\delta_a \wedge \nu^2}{4\tau}|z_i - z_{i-1}|_V^2 + \frac{\tau}{\delta_a \wedge \nu^2}|k_i|_{V^*}^2.
  \end{gather}
  From \eqref{P_TD_est1} and \eqref{P_TD_est11}, we can deduce that
  \begin{align}
    &\quad \frac{\delta_a \wedge \nu^2}{4\tau}\sum_{i=1}^{n_\tau}|z_i - z_{i-1}|_V^2
    \\
    &\leq \frac{4((C_V^{L^4})^4 + 1)}{\delta_a \wedge \nu^2} \cdot \tau \sum_{i=1}^{n_\tau}(|b_i|_H^2 + |c_i|_H^2 + |\omega_i|_{[L^\infty(\Omega)]^N}^2 + |A_i|_{[L^\infty(\Omega)]^{N \times N}}^2) \cdot
    \\
    &\qquad \cdot (|z_i|_V^2 + |p_i|_V^2 + |z_{i-1}|_V^2) + \frac{\tau}{\delta_a \wedge \nu^2}\sum_{i=1}^{n_\tau} |k_i|_{V^*}^2
  \end{align}
  \begin{align}
    &\leq \frac{16((C_V^{L^4})^4 + 1)(T + 1)}{\delta_a \wedge \nu^2} (|b|_\sH^2 + |c|_\sH^2 + |\omega|_{[L^\infty(Q)]^N}^2 + |A|_{[L^\infty(Q)]^{N \times N}}^2) \cdot
    \\
    &\quad \cdot \frac{1 + \mu^2 + \nu^2 + (C_V^{L^4})^2|a(0)|_H}{1 \wedge \delta_a \wedge \mu^2 \wedge \nu^2} e^{4C_1(T+1)} (|p_0|_V^2 + |z_0|_V^2 + |h|_\sH^2 + |k|_\sH^2) + \frac{1}{\delta_a \wedge \nu^2}|k|_{\sV^*}^2
    \\
    &\leq \frac{17((C_V^{L^4})^4 + 1)(T + 1)(1 + \mu^2 + \nu^2 + (C_V^{L^4})^2|a(0)|_H)}{1 \wedge \delta_a^2 \wedge \mu^4 \wedge \nu^4} \cdot
    \\
    &\quad \cdot e^{5 C_1 (T+1)} (|A|_{[L^\infty(Q)]^{N \times N}}^2 + 1) (|p_0|_V^2 + |z_0|_V^2 + |h|_{\sV^*}^2 + |k|_{\sV^*}^2).
  \end{align}
  Thus, the estimates \eqref{P_TD_est2} and \eqref{P_TD_est3} are obtained with the constant
  \begin{equation}
    C_2 = \frac{17((C_V^{L^4})^4 + 1)(T + 1)(1 + \mu^2 + \nu^2 + (C_V^{L^4})^2|a(0)|_H)}{1 \wedge \delta_a^2 \wedge \mu^4 \wedge \nu^4}.
  \end{equation}
\end{proof}

\begin{proof}[Proof of Key-Lemma \ref{lem:Psol}]
  Owing to Key-Lemma \ref{lem:P_TD} and \ref{lem:P_est}, we can derive the following 
  \\
  boundedness:
  \begin{itemize}
    \item $\{ [p]_\tau \}_{\tau \in (0,\tau_2)}, \{ [z]_\tau \}_{\tau \in (0,\tau_2)}$ are bounded in $W^{1,2}(0,T;V)$,
    \item $\{ [\overline{p}]_\tau \}_{\tau \in (0,\tau_2)}, \{ [\underline{p}]_\tau \}_{\tau \in (0,\tau_2)}, \{ [\overline{z}]_\tau \}_{\tau \in (0,\tau_2)}, \{ [\underline{z}]_\tau \}_{\tau \in (0,\tau_2)}$ are bounded in $L^\infty(0,T;V)$.
  \end{itemize}
  Therefore, by applying the compactness theory of Aubin's type (cf. \cite[Corollary 4]{MR0916688}), there exists a subsequence $\{ \tau_n \}_{n=1}^\infty \subset \{ \tau \}$ and there exists a pair of functions $[p,z]$ satisfying the following convergence:
  \begin{gather}
    p_n := [p]_{\tau_n} \to p, \, z_n := [z]_{\tau_n} \to z \mbox{ in } C([0,T];H) \mbox{ and weakly in } W^{1,2}(0,T;V)
    \\
    \mbox{ as } n \to \infty, \label{P_sol_conv1}
  \end{gather}
  and in particular,
  \begin{equation}
    \lim_{n \to \infty} [p_n(0), z_n(0)] = \lim_{n \to \infty} [p_0, z_0] = [p_0, z_0] \mbox{ in } [H]^2.
  \end{equation}
  Moreover, noting the following:
  \begin{align}
    &\bigl( |\overline{p}_{\tau_n} - p_n|_V \wedge |\underline{p}_{\tau_n} - p_n|_V \wedge |\overline{z}_{\tau_n} - z_n|_V \wedge |\underline{z}_{\tau_n} - z_n|_V \bigr)(t)
    \\
    &\quad \leq \int_{(t_{i-1}, t_i) \cap (0,T)} \bigl( |\partial_t p_\tau(t)|_V \wedge |\partial_t z_\tau(t)|_V \bigr)\,dt \leq \tau^\frac{1}{2} \bigl( |\partial_t p_\tau|_\sV \wedge |\partial_t z_\tau|_\sV \bigr),
    \\
    &\quad \mbox{ for } t \in (t_{i-1}, t_i) \cap (0,T), \ i = 1,2,3,\dots, n_\tau, \mbox{ and } \tau \in (0,\tau_2),
  \end{align}
  we can derive additional convergences as follows:
  \begin{gather}
    \overline{p}_n := [\overline{p}]_{\tau_n} \to p, \, \underline{p}_n := [\underline{p}]_{\tau_n} \to p, \, \overline{z}_n := [\overline{z}]_{\tau_n} \to z, \, \underline{z}_n := [\underline{z}]_{\tau_n} \mbox{ in } L^\infty(0,T;H)
    \\
    \mbox{ and weakly-$*$ in } L^\infty(0,T;V), \mbox{ as } n \to \infty. \label{P_sol_conv2}
  \end{gather}
  In particular, it holds that
  \begin{gather}
    p_n(t) \to p(t), \ \overline{p}_n(t) \to p(t), \ \underline{p}_n(t) \to p(t), \ z_n(t) \to z(t), \ \overline{z}_n(t) \to z(t), \ \underline{z}_n(t) \to z(t),
    \\
    \mbox{ in } H \mbox{ and weakly in } V, \mbox{ as } n \to \infty, \mbox{ for any } t \in [0,T].
  \end{gather}

  Now, we let:
  \begin{equation}
    \left\{ \begin{aligned}
      &\overline{a}_n := [\overline{a}]_{\tau_n}, \, \overline{b}_n := [\overline{b}]_{\tau_n}, \, \overline{c}_n := [\overline{c}]_{\tau_n}, \, \overline{\lambda}_n := [\overline{\lambda}]_{\tau_n}, \, \overline{\xi}_n := [\overline{\xi}]_{\tau_n}, \, 
      \\
      &\overline{\omega}_n := [\overline{\omega}]_{\tau_n}, \overline{A}_n := [\overline{A}]_{\tau_n}, \, \overline{h}_n := [\overline{h}]_{\tau_n}, \, \overline{k}_n := [\overline{k}]_{\tau_n}.
    \end{aligned} \right.
  \end{equation}
  Then, the sequences solve the following variational system for any open interval $I \subset (0,T)$ and any $n = 1,2,3,\dots$:
  \begin{gather}
    \int_I (\partial_t p_n(t) + \overline{\omega}_n(t) \cdot \nabla \overline{z}_n(t), \varphi)_H \,dt + \int_I \int_\Omega (\overline{\lambda}_n \overline{p}_n + \overline{\xi}_n \underline{z}_n)(t) \varphi \,dxdt
    \\
    + \int_I (\nabla (\overline{p}_n + \mu^2 \partial_t p_n)(t), \nabla \varphi)_{[H]^N}\,dt = \int_I \langle [\overline{h}]_n(t), \varphi \rangle_V, \mbox{ for any } \varphi \in V,
  \end{gather}
  and
  \begin{gather}
    \int_I (\overline{a}_n \partial_t z_n + \overline{b}_n \overline{z}_n + \overline{c}_n \underline{p}_n)(t) \psi \,dxdt + \int_I (\overline{A}_n(t) \nabla \overline{z}_n(t) + \overline{p}_n(t)\overline{\omega}_n(t), \nabla \psi)_{[H]^N} \,dt
    \\
    + \nu^2 \int_I (\nabla \partial_t z_n(t), \nabla \psi)_{[H]^N}^2 \,dt = \int_I \langle \overline{k}_n(t), \psi \rangle_V \,dt, \ \mbox{ for any } \psi \in V.
  \end{gather}
  By noting the continuous embedding from $V$ to $L^4(\Omega)$ as in Remark \ref{emb} and the following convergences as $n \to \infty$:
  \begin{gather}
    \left\{ \begin{aligned}
      &\bullet \overline{\omega}_n \varphi \to \omega \varphi, \ \overline{\omega}_n \psi \to \omega \psi, \ \overline{A}_n \nabla \psi \to A \nabla \psi \mbox{ in } [\sH]^N,
      \\
      &\bullet \overline{\lambda}_n \varphi \to \lambda \varphi, \ \overline{\xi}_n \varphi \to \xi \varphi, \ \overline{a}_n \psi \to a \psi, \ \overline{b}_n \psi \to b \psi, \ \overline{c}_n \psi \to c \psi \mbox{ in } L^2(0,T;L^\frac{4}{3}(\Omega)),
    \end{aligned} \right.
    \\
    \mbox{ for any } \varphi \in V \mbox{ and } \psi \in V,
  \end{gather}
  we have
  \begin{gather}
    \int_I (\partial_t p(t) + \omega(t) \cdot \nabla z(t), \varphi)_H \,dt + \int_I \int_\Omega (\lambda p + \xi z)(t) \varphi \,dxdt
    \\
    + \int_I (\nabla (p + \mu^2 \partial_t p)(t), \nabla \varphi)_{[H]^N}\,dt = \int_I \langle h(t), \varphi \rangle_V, \mbox{ for any } \varphi \in V,
  \end{gather}
  and
  \begin{gather}
    \int_I (a \partial_t z + b z + c p)(t) \psi \,dxdt + \int_I (A(t) \nabla z(t) + p(t) \omega(t), \nabla \psi)_{[H]^N} \,dt
    \\
    + \nu^2 \int_I (\nabla \partial_t z(t), \nabla \psi)_{[H]^N}^2 \,dt = \int_I \langle k(t), \psi \rangle_V \,dt, \ \mbox{ for any } \psi \in V.
  \end{gather}
  Since the open interval $I \subset (0,T)$ is arbitrary, the limiting pair $[p,z]$ solves the variational system (P), and thus, we complete the proof of Key-Lemma \ref{lem:Psol}.
\end{proof}

\begin{keylem}[Uniqueness of solution to (P) and continuous dependence] \label{lem:P_uni}
  Let \linebreak $[a^i, b^i, c^i, \lambda^i, \xi^i, \omega^i, A^i] \in \sS$, $[p_0^i, z_0^i] \in [V]^2$, $[h^i, k^i] \in [\sV^*]^2$, and let $[p^i, z^i] \in [W^{1,2}(0,T;V)]^2$ be the solution to (P) corresponding to the septuplet $[a^i, b^i, c^i, \lambda^i, \xi^i, \omega^i, A^i]$, the initial pair $[p_0^i, z_0^i]$ and the forcing $[h^i,k^i]$ for $i=1,2$. In addition, we set:
  \begin{gather}
    J(t) := |(p^1 - p^2)(t)|_H^2 + \mu^2 |\nabla (p^1 - p^2)(t)|_{[H]^N}^2 + |\sqrt{a^1(t)}(z^1 - z^2)(t)|_H^2 
    \\
    + \nu^2 |\nabla(z^1 - z^2)(t)|_{[H]^N}^2, \mbox{ for any } t \in [0,T], \label{P_uni00}
  \end{gather}
  and
  \begin{align}
    R_0(t) &:= \biggl( \left( \frac{12((C_V^{L^4})^2 + 1)}{1 \wedge \delta_a \wedge \mu^2 \wedge \nu^2} \right) \cdot 
    \\
    &\quad \cdot \bigl(|\lambda^1(t)|_H + |\xi^1(t)|_H + |\partial_t a^1(t)|_H + |b^1(t)|_H + |c^1(t)|_H + |\omega^1(t)|_{[L^\infty(\Omega)]^N} + 1 \bigr) \biggr).
    \\
    R_1(t) &:= |p^2(t)|_V^2 (|(\lambda^1 - \lambda^2)(t)|_H^2 + |(c^1 - c^2)(t)|_H^2) \label{P_uni01}
    \\
    &\quad + |z^2(t)|_V^2 (|(\xi^1 - \xi^2)(t)|_H^2 + |(a^1 - a^2)(t)|_H^2 + |(b^1 - b^2)(t)|_H^2)
    \\
    &\quad + |\nabla z^2(t) \cdot (\omega^1 - \omega^2)(t)|_H^2 + |(A^1 - A^2)(t) \nabla z^2(t)|_{[H]^N}^2 + |p^2(t) (\omega^1 - \omega^2)(t)|_{[H]^N}^2,
    \\
    &\qquad\qquad \mbox{ for a.e. } t \in (0,T).
  \end{align} 
  Then, the following estimate holds:
  \begin{gather}
    \frac{d}{dt} J(t) \leq R_0(t) J(t) + |(h^1 - h^2)(t)|_{V^*}^2 + |(k^1 - k^2)(t)|_{V^*}^2 + ((C_V^{L^4})^2 + 1) R_1(t),
    \\
    \mbox{ for a.e. } t \in (0,T). \label{P_uni99}
  \end{gather}
  In particular, the solution to (P) is unique.
\end{keylem}

\begin{proof}
  By letting $\varphi = (p^1 - p^2)(t)$ in the variational identity \eqref{P_var1}, taking the difference of the both side, and using the continuous embedding from $V$ to $L^4(\Omega)$, it is computed that:
  \begin{align}
    &\quad \frac{1}{2}\frac{d}{dt}\bigl( |(p^1 - p^2)(t)|_H^2 + \mu^2 |\nabla (p^1 - p^2)(t)|_{[H]^N}^2 \bigr) \label{P_uni_1}
    \\
    &\leq (C_V^{L^4})^2 (|\lambda^1(t)|_H + 1) |(p^1 - p^2)(t)|_V^2 + \frac{(C_V^{L^4})^2}{2}|p^2(t)|_V^2 |(\lambda^1 - \lambda^2)(t)|_H^2 
    \\
    &\quad + \frac{(C_V^{L^4})^2}{2}(|\xi^1(t)|_H + 1) (|(p^1 - p^2)(t)|_V^2 + |(z^1 - z^2)(t)|_V^2) 
    \\
    &\quad + \frac{(C_V^{L^4})^2}{2}|z^2(t)|_V^2 |(\xi^1 - \xi^2)(t)|_H^2 
    \\
    &\quad + \frac{1}{2}(|\omega^1(t)|_{[L^\infty(\Omega)]^N} + 1) \bigl( |\nabla (z^1 - z^2)(t)|_{[H]^N}^2 + |(p^1 - p^2)(t)|_H^2 \bigr)
    \\
    &\quad + |z^2(t) \cdot \nabla (\omega^1 - \omega^2)(t)|_H^2+ \frac{1}{2}(|(p^1 - p^2)(t)|_V^2 + |(h^1 - h^2)(t)|_{V^*}^2), \mbox{ for a.e. } t \in (0,T).
  \end{align}
  On the other hand, by letting $\psi = (z^1 - z^2)(t)$ in \eqref{P_var2}, taking the difference between two variational equalities, using the continuous embedding from $V$ to $L^4(\Omega)$ and invoking that $A(t,x)$ is positive for a.e. $(t,x) \in Q$, we can calculate that
  \begin{align}
    &\quad \int_\Omega (a^1 \partial_t z^1 - a^2 \partial_t z^2)(t) (z^1 - z^2)(t)\,dx + \frac{\nu^2}{2} \frac{d}{dt} |\nabla (z^1 - z^2)(t)|_{[H]^N}^2 \label{P_uni_2}
    \\
    &\leq (C_V^{L^4})^2(|b^1(t)|_H + 1) |(z^1 - z^2)(t)|_V^2 + \frac{(C_V^{L^4})^2}{2}|z^2(t)|_V^2 |(b^1 - b^2)(t)|_H^2 
    \\
    &\quad + \frac{(C_V^{L^4})^2}{2}(|c^1(t)|_H + 1) (|(p^1 - p^2)(t)|_V^2 + |(z^1 - z^2)(t)|_V^2)
    \\
    &\quad + \frac{(C_V^{L^4})^2}{2}|p^2(t)|_V^2 |(c^1 - c^2)(t)|_H^2 + \frac{1}{2}|\nabla (z^1 - z^2)(t)|_{[H]^N}^2 + \frac{1}{2}|(A^1 - A^2)(t) \nabla z^2(t)|_{[H]^N}^2
    \\
    &\quad + \frac{(|\omega^1(t)|_{[L^\infty(\Omega)]^N} + 1)}{2} \bigl( |\nabla (z^1 - z^2)(t)|_{[H]^N}^2 + |(p^1 - p^2)(t)|_H^2 \bigr)
    \\
    &\quad + \frac{1}{2}|p^2(t)(\omega^1 - \omega^2)(t)|_{[H]^N}^2 + \frac{1}{2}|(z^1 - z^2)(t)|_V^2 + \frac{1}{2}|(k^1 - k^2)(t)|_{V^*}^2, \mbox{ for a.e. } t \in (0,T).
  \end{align} \noeqref{P_uni_2}
  The first term can be computed by
  \begin{align}
    &\int_\Omega (a^1 \partial_t z^1 - a^2 \partial_t z^2)(t) (z^1 - z^2)(t)\,dx \label{P_uni_3}
    \\
    &\quad = \frac{1}{2} \frac{d}{dt} |\sqrt{a^1(t)} (z^1 - z^2)(t)|_H^2 - \frac{1}{2} \int_\Omega \partial_t a(t)|(z^1 - z^2)(t)|^2 \,dx 
    \\
    &\qquad + \int_\Omega (a^1 - a^2)(t) \partial_t z^2(t) (z^1 - z^2)(t) \,dx
    \\
    &\quad \geq \frac{1}{2} \frac{d}{dt} |\sqrt{a^1(t)} (z^1 - z^2)(t)|_H^2 - \frac{(C_V^{L^4})^2}{2} |\partial_t a^1(t)|_H^2 |(z^1 - z^2)(t)|_V^2
    \\
    &\qquad - \frac{(C_V^{L^4})^2}{2}|(z^1 - z^2)(t)|_V^2 - \frac{1}{2}|\partial_t z^2|_V^2 |(a^1 - a^2)(t)|_H^2.
  \end{align}
  By integrating \eqref{P_uni_1}--\eqref{P_uni_3}, we arrive at the Gronwall type inequality \eqref{P_uni99}.
\end{proof}

According to Key-Lemma \ref{lem:P_est}, \ref{lem:P_uni}, and the convergences \eqref{P_sol_conv1} and \eqref{P_sol_conv2}, we immediately derive the following Corollary.

\begin{cor}\label{cor:P_est}
  Let $[p,z]$ be the solution to the system (P) corresponding to a septuplet $[a,b,c,\lambda,\xi,\omega,A] \in \sS$, initial data $[p_0,z_0] \in [V]^2$ and forcing pair $[h,k] \in [\sV^*]^2$. Then, the solution $[p,z]$ fulfills the following estimates:
    \begin{align}
      |p|_{C([0,T];V)}^2 + |z|_{C([0,T];V)}^2 &\leq \frac{1 + \mu^2 + \nu^2 + (C_V^{L^4})^2|a(0)|_H}{1 \wedge \delta_a \wedge \mu^2 \wedge \nu^2}e^{4 C_1 (T+1)}\cdot
      \\
      &\qquad \cdot (|p_0|_V^2 + |z_0|_V^2 + |h|_{\sV^*}^2 + |k|_{\sV^*}^2), \label{P_est_97}
    \end{align}
    \begin{gather}
      \frac{1 \wedge \mu^2}{4}|\partial_t p|_{\sV}^2 \leq C_2 e^{5C_1 (T + 1)} (|p_0|_V^2 + |z_0|_V^2 + |h|_{\sV^*}^2 + |k|_{\sV^*}^2), \label{P_TD_est98}
    \end{gather}
    and
    \begin{align}
      &\frac{\delta_a \wedge \nu^2}{4}|\partial_t z|_{\sV}^2 \leq C_2 e^{5 C_1 (T + 1)}(|A|_{[L^\infty(Q)]^{N \times N}}^2 + 1) (|p_0|_V^2 + |z_0|_V^2 + |h|_{\sV^*}^2 + |k|_{\sV^*}^2), \label{P_TD_est99}
    \end{align}
    where the constants $C_1, C_2$ are as in Key-Lemma \ref{lem:P_est}.
\end{cor} 

Now, we prepare additional notations. Owing to Key-Lemma \ref{lem:Psol}, we can define an operator $\sP = \sP(a,b,c,\lambda,\xi,\omega,A):[V]^2 \times [\sV^*]^2 \longrightarrow \sZ$ for any septuplet $[a,b,c,\lambda,\xi,\omega,A] \in \sS$, as follows:
\begin{align}
  &\sP = \sP(a,b,c,\lambda,\xi,\omega,A):([p_0, z_0],[h,k]) \in [V]^2 \times [\sV^*]^2 
  \\
  &\qquad \mapsto \sP([p_0, z_0],[h,k]) := [p,z] \in \sZ\mbox{: the unique solution to (P)}. \label{sol_op_P}
\end{align}

\begin{rem} \label{rem:p*}
  Set $\sT$ as an isomorphism defined as:
  \begin{equation}
    \sT \varphi(t) = \varphi(T-t) \mbox{ in } H, \mbox{ for a.e. } t \in (0,T).
  \end{equation}
  For any $\varepsilon \in (0,1)$, let $[u_\varepsilon^*, v_\varepsilon^*]$ be the optimal control of the problem (OCP)$_\varepsilon$, and let $[\eta_\varepsilon,\theta_\varepsilon]$ be the solution to the state system (S)$_\varepsilon$ corresponding to the initial data $[\eta_0, \theta_0]$ and the forcings $[u_\varepsilon^*, v_\varepsilon^*]$. Also, let us set $\sP_\varepsilon^\circ \in \sL([\sH]^2)$ as the restriction $\sP|_{[0,0] \times [\sH]^2}$ of the bounded linear operator $\sP = \sP(a,b,c,\lambda,\omega,A)$ in the case:
  \begin{equation}
    \left\{ \begin{aligned}
      &a = \sT[-\alpha_0(\eta_\varepsilon)],
      \\
      &b = \sT[-\alpha_0'(\eta_\varepsilon)\partial_t \eta_\varepsilon],
      \\
      &c = 0,
      \\
      &\lambda = \sT[g'(\eta_\varepsilon) + \alpha''(\eta_\varepsilon) \gamma_\varepsilon(\nabla \theta_\varepsilon)],
      \\
      &\xi = \sT[\alpha_0'(\eta_\varepsilon) \partial_t \theta_\varepsilon],
      \\
      &\omega = \sT[\alpha'(\eta_\varepsilon) \nabla \gamma_\varepsilon(\nabla \theta_\varepsilon)],
      \\
      &A = \sT[\alpha(\eta_\varepsilon) \nabla^2 \gamma_\varepsilon(\nabla \theta_\varepsilon)],
    \end{aligned} \right. \mbox{ in } [\sH]^7.
  \end{equation}
  Based on these settings, we set a bounded linear operator $\sP_\varepsilon^*$ as:
  \begin{equation}
    \sP_\varepsilon^* := \sT \circ \sP_\varepsilon^\circ \circ \sT, \mbox{ in } \sL([\sH]^2).
  \end{equation}
  Then, the solution $[p_\varepsilon^*,z_\varepsilon^*]$ to the variational system \eqref{adj_1}--\eqref{adj_3} can be represented by:
  \begin{equation}
    [p_\varepsilon^*, z_\varepsilon^*] = \sP_\varepsilon^*(M_\eta(\eta_\varepsilon - \eta_{\rm ad}),M_\theta(\theta_\varepsilon - \theta_{\rm ad})) \mbox{ in } [\sH]^2.
  \end{equation}
\end{rem}
\noeqref{adj_2}

In the light of Key-Lemma \ref{lem:Psol}, \ref{lem:P_uni} and Corollary \ref{cor:P_est}, we can obtain the following results for the continuous dependence of solution with respect to septuplet, initial data, and forcing pair.

\begin{cor}\label{cor:CD}
  We assume that $[a, b, c, \lambda, \xi, \omega, A] \in \sS$ and $\{ [a_n, b_n, c_n, \lambda_n, \xi_n, \omega_n, A_n] \}_{n=1}^\infty$ $\subset \sS$ fulfill the following convergences:
  \begin{gather}
    \left\{ \begin{aligned}
      &\bullet \ [b_n,c_n,\lambda_n,\xi_n] \to [b,c,\lambda,\xi] \mbox{ in } [\sH]^4,
      \\
      &\bullet \ a_n \to a \mbox{ in } \sH \mbox{ and weakly in } W^{1,2}(0,T;H),
      \\
      &\bullet \ \omega_n \to \omega \mbox{ weakly-$*$ in } [L^\infty(Q)]^N, \mbox{ and in the pointwise sense, a.e. in } Q,
      \\
      &\bullet \ A_n \to A \mbox{ weakly-$*$ in } [L^\infty(Q)]^{N \times N}, \mbox{ and in the pointwise sense, a.e. in } Q,
    \end{aligned} \right. \label{CD_01}
    \\
    \mbox{ as } n \to \infty.
  \end{gather} 
  Also, we let $[p_0,z_0] \in [V]^2$, $\{ [p_{0,n}, z_{0,n}] \}_{n=1}^\infty \subset [V]^2$, $[h,k] \in [\sV^*]^2$, and $\{ [h_n, k_n] \}_{n=1}^\infty \subset [\sV^*]^2$. Besides, we set $[p_n, z_n] := \sP(a_n, b_n, c_n, \lambda_n, \xi_n, \omega_n, A_n)([p_{0,n}, z_{0,n}], [h_n,k_n])$ for $n = 1,2,3,\dots$, and $[p,z] := \sP(a, b, c, \lambda, \xi, \omega, A)([p_0,z_0],[h,k])$. Then, if the sequences 
  \\
  $\{ [p_{0,n},z_{0,n}] \}_{n=1}^\infty$ and $\{ [h_n,k_n] \}_{n=1}^\infty$ satisfy the following convergences:
  \begin{equation}
    [[p_{0,n},z_{0,n}], [h_n,k_n]] \to [[p_0,z_0], [h,k]] \mbox{ weakly in } [V]^2 \times [\sV^*]^2 \mbox{ as } n \to \infty, \label{CD_02_0}
  \end{equation}
  the following convergences hold:
  \begin{equation}
    p_n \to p, \, z_n \to z \mbox{ in } C([0,T];H) \mbox{ and weakly in } W^{1,2}(0,T;V) \mbox{ as } n \to \infty. \label{CD_02}
  \end{equation}
  Moreover, if we suppose:
  \begin{equation}
    [[p_{0,n},z_{0,n}], [h_n,k_n]] \to [[p_0,z_0], [h,k]] \mbox{ in } [V]^2 \times [\sV^*]^2 \mbox{ as } n \to \infty, \label{CD_03_0}
  \end{equation}
  then, it holds that:
  \begin{equation}
    p_n \to p, \, z_n \to z \mbox{ in } C([0,T];V) \mbox{ as } n \to \infty. \label{CD_03}
  \end{equation}
\end{cor}
\begin{proof}
  By virtue of Corollary \ref{cor:P_est} and the assumption \eqref{CD_01}, we can see that the sequence $\{ [p_n, z_n] \}_{n=1}^\infty$ is bounded in $[W^{1,2}(0,T;V)]^2$. Hence, by applying compactness theory of Aubin's type again, we can find a subsequence $\{ n_j \}_{j=1}^\infty \subset \{ n \}$ with a limiting pair $[\tilde p, \tilde z] \in [\sH]^2$ satisfying the following convergence:
  \begin{gather}
    p_{n_j} \to \tilde p, \, z_{n_j} \to \tilde z \mbox{ in } C(0,T;H) \mbox{ and weakly in } W^{1,2}(0,T;V) \mbox{ as } j \to \infty.
  \end{gather}
  Furthermore, by the same argument as in the proof of Key-Lemma \ref{lem:Psol} and uniqueness result, obtained in Key-Lemma \ref{lem:P_uni}, we can check that $[\tilde p, \tilde z] = \sP(a,b,c,\lambda,\xi,\omega,A)([p_0,z_0],[h,k]) = [p,z]$. Finally, by the uniqueness of the limit, the convergence \eqref{CD_02} is verified without taking a subsequence. 
  
  Next, we confirm the convergence \eqref{CD_03}. By applying Key-Lemma \ref{lem:P_uni} under:
  \begin{gather}
    \left\{ \begin{aligned}
      &[a^1, b^1, c^1, \lambda^1, \xi^1, \omega^1, A^1] = [a_n, b_n, c_n, \lambda_n, \xi_n, \omega_n, A_n],
      \\
      &[p_0^1, z_0^1] = [p_{0,n}, z_{0,n}], \, [h^1, k^1] = [h_n, k_n], \mbox{ and }
    \end{aligned} \right.
    \\
    \left\{ \begin{aligned}
      &[a^2, b^2, c^2, \lambda^2, \xi^2, \omega^2, A^2] = [a, b, c, \lambda, \xi, \omega, A],
      \\
      &[p_0^2, z_0^2] = [p_0, z_0], \, [h^2, k^2] = [h, k],
    \end{aligned} \right.
  \end{gather}
  Gronwall's lemma allows us to find the following estimate:
  \begin{gather}
    J_n(t) \leq R^*\biggl( J_n(0) + |h_n - h|_{\sV^*}^2 + |k_n - k|_{\sV^*}^2 + ((C_V^{L^4})^2 + 1)\int_0^T R_n^*(t)\,dt \biggr),
    \\
    \mbox{ for any } t \in [0,T], \label{CD_04}
  \end{gather}
  where
  \begin{align}
    R^* &:= \exp\Biggl( \left( \frac{12(T + 1)((C_V^{L^4})^2 + 1)}{1 \wedge \delta_a \wedge \mu^2 \wedge \nu^2} \right) \cdot 
    \\
    &\qquad \cdot \sup_{n \in \N} \bigl\{ |\lambda_n|_\sH + |\xi_n|_\sH + |\partial_t a_n|_\sH + |b_n|_\sH + |c_n|_\sH + |\omega_n|_{L^\infty(Q)} + 1 \bigr\} \Biggr), \label{CD_05}
  \end{align}
  \begin{gather}
    J_n(t) := |(p_n - p)(t)|_H^2 + \mu^2 |\nabla (p_n - p)(t)|_{[H]^N}^2 + |\sqrt{a_n(t)}(z_n - z)(t)|_H^2 
    \\
    + \nu^2 |\nabla(z_n - z)(t)|_{[H]^N}^2, \mbox{ for any } t \in [0,T], \label{CD_06}
  \end{gather}
  and
  \begin{align}
    R_n^*(t) &:= |p(t)|_V^2 (|(\lambda_n - \lambda)(t)|_H^2 + |(c_n - c)(t)|_H^2) \label{CD_07}
    \\
    &\quad + |z(t)|_V^2 (|(\xi_n - \xi)(t)|_H^2 + |(a_n - a)(t)|_H^2 + |(b_n - b)(t)|_H^2)
    \\
    &\quad + |\nabla z(t) \cdot (\omega_n - \omega)(t)|_H^2 + |(A_n - A)(t) \nabla z(t)|_{[H]^N}^2 + |p(t) (\omega_n - \omega)(t)|_{[H]^N}^2,
    \\
    &\qquad\qquad \mbox{ for a.e. } t \in (0,T). 
  \end{align}
  Now, \eqref{CD_01} implies that $R^*$ is a finite constant, and $R_n^*$ converges to $0$ in $L^1(0,T)$ as $n \to \infty$. Based on this, letting $n \to \infty$ yields the convergence \eqref{CD_03}.

  Thus, we complete the proof of Corollary \ref{cor:CD}.
\end{proof}

\begin{cor}[Linearity and boundedness of solution operator] \label{cor_bddness}
  Let us denote 
  \\
  $\sZ = [W^{1,2}(0,T;V)]^2$, and let $\sZ$ be a Banach space endowed with a norm $\| \cdot \|: \sZ \longrightarrow [0,\infty)$, defined as:
    \begin{equation}
      \| [p,z] \| := |[p,z]|_{[W^{1,2}(0,T;V)]^2} + |[p,z]|_{C([0,T];[V]^2)}.
    \end{equation}
    We let $[a,b,c,\lambda,\xi,\omega,A] \in \sS$, $[p_0, z_0] \in [V]^2$, $[h,k] \in [\sV^*]^2$, and $[p,z] := \sP([p_0,z_0],[h,k])$ $= \sP(a,b,c,\lambda,\xi,\omega,A)([p_0,z_0],[h,k])$. Then, there exist two positive constants $M_0^* = M_0^*(a,b,c,\lambda,\xi,\omega,A)$, $M_1^* = M_1^*(a,b,c,\lambda,\xi,\omega,A)$ such that:
    \begin{gather}
      M_0^* |[[p_0,z_0],[h,k]]|_{[V]^2 \times [\sV^*]^2} \leq \|[p,z]\| \leq M_1^* |[[p_0,z_0],[h,k]]|_{[V]^2 \times [\sV^*]^2}, 
      \\
      \mbox{ for any } [p_0, z_0] \in [V]^2, \mbox{ and } [h,k] \in [\sV^*]^2. \label{sP_bdd}
    \end{gather}
    Moreover, $\sP$ is an bijective operator from $[V]^2 \times [\sV^*]^2$ to $\sZ$, Therefore, $\sP$ is an isomorphism between a Hilbert space $[V]^2 \times [\sV^*]^2$ and a Banach space $(\sZ, \| \cdot \|)$.
\end{cor}
\begin{proof}
  Corollary \ref{cor:P_est} implies that $\sP$ is a linear bounded operator from $[V]^2 \times [\sV^*]^2$ to $(\sZ, \| \cdot \|)$ with the constant:
  \begin{equation}
    M_1^* := \left(\frac{48(1 + \mu^2 + \nu^2 + (C_V^{L^4})^2)|a(0)|_H}{1 \wedge \delta_a \wedge \mu^2 \wedge \nu^2} C_2 e^{5 C_1 (T+1)}(|A|_{[L^\infty(Q)]^{N \times N}} + 1)\right)^\frac{1}{2}.
  \end{equation}

  Now, we define a linear operator $\sQ = \sQ(a,b,c,\lambda,\xi,\omega,A): \sZ \longrightarrow [\sV^*]^2$ which maps $[p,z]$ to a pair of bounded linear functionals $[\tilde h, \tilde k] =: \sQ(p,z) \in [\sV^*]^2$, given as:
  \begin{gather}
    \langle \tilde h, \varphi \rangle_\sV := \int_0^T (\partial_t p(t) + \omega(t) \cdot \nabla z(t), \varphi(t))_H \,dt + \int_0^T \int_\Omega (\lambda(t) p(t) + \xi(t)z(t)) \varphi(t) \,dxdt 
    \\
    + \int_0^T (\nabla (p + \mu^2 \partial_t p)(t), \nabla \varphi(t))_{[H]^N} \,dt, \mbox{ for any } \varphi \in \sV,
  \end{gather}
  and
  \begin{gather}
    \langle \tilde k,\psi \rangle_\sV := \int_0^T \int_\Omega (a(t) \partial_t z(t) + b(t) z(t) + c(t) p(t)) \psi(t) \,dxdt 
    \\
    + \int_0^T (A(t) \nabla z(t) + \nu^2\nabla \partial_t z(t) + p(t)\omega(t), \nabla \psi(t))_{[H]^N} \,dt, \mbox{ for any } \psi \in \sV.
  \end{gather}
  Then, one can see that $\sQ$ is a linear bounded operator from $\sZ$ to $[\sV^*]^2$ together with the following computations:
  \begin{align}
    &\quad |\langle \tilde h, \varphi \rangle_\sV| 
    \\
    &\leq (|\partial_t p|_\sH + |\omega|_{[L^\infty(Q)]^N} |\nabla z|_{[\sH]^N})|\varphi|_\sH + (C_V^{L^4})^2(|\lambda|_\sH |p|_{C([0,T];V)} + |\xi|_\sH |z|_{C([0,T];V)}) |\varphi|_\sV
    \\
    &\qquad + (|\nabla p|_{[\sH]^N} + \mu^2 |\nabla \partial_t p|_{[\sH]^N}) |\nabla \varphi|_{[\sH]^N} 
    \\
    &\leq 2 (1 + \mu^2)\bigl( (C_V^{L^4})^2 (|\lambda|_\sH + |\xi|_\sH) + |\omega|_{[L^\infty(Q)]^N} + 1\bigr)\| [p,z] \| \cdot |\varphi|_\sV,
    \\
    &\qquad\qquad\qquad\qquad\qquad\qquad \mbox{ for any } \varphi \in \sV,
  \end{align}
  and
  \begin{align}
    &\quad|\langle \tilde k,\psi \rangle_\sV| 
    \\
    &\leq (C_V^{L^4})^2(|a|_{C([0,T];H)} |\partial_t z|_\sV + |b|_\sH |z|_{C([0,T];V)} + |c|_\sH |p|_{C([0,T];V)})|\psi|_\sV 
    \\
    &\qquad + (|A|_{[L^\infty(Q)]^{N \times N}} |\nabla z|_{[\sH]^N} + \nu^2 |\nabla \partial_t z|_{[\sH]^N} + |\omega|_{[L^\infty(Q)]^N} |p|_\sH) |\nabla \psi|_{[\sH]^N}
    \\
    &\leq 2 (1 + \nu^2) \bigl( (C_V^{L^4})^2 (|a|_{C([0,T];H)} + |b|_\sH + |c|_\sH) + |A|_{[L^\infty(Q)]^{N \times N}} + |\omega|_{[L^\infty(Q)]^N} + 1\bigr)\cdot 
    \\
    &\qquad \cdot \| [p,z] \| \cdot |\psi|_\sV.
  \end{align}
  and hence,
  \begin{align}
    &\quad |\sQ(p,z)|_{[\sV^*]^2} \leq \sqrt{2}(|\tilde h|_{\sV^*} + |\tilde k|_{\sV^*}) \label{P_bdd_1}
    \\
    &\leq \tilde C_4 (|a|_{C([0,T];H)} + |b|_\sH + |c|_\sH + |\lambda|_\sH + |\xi|_\sH + |A|_{[L^\infty(Q)]^{N \times N}} + |\omega|_{[L^\infty(Q)]^N} + 1) \cdot
    \\
    &\qquad \cdot \| [p,z] \|,
  \end{align}
  with
  \begin{equation}
    \tilde C_4 := 4\sqrt{2} (1 + \mu^2 + \nu^2)(1 + (C_V^{L^4})^2).
  \end{equation}

  Besides, we define a linear operator $\overline{\sP} = \overline{\sP}(a,b,c,\lambda,\xi,\omega,A): \sZ \longrightarrow [V]^2 \times [\sV^*]^2$, by letting:
  \begin{equation}
    \overline{\sP}: [p,z] \in \sZ \mapsto \overline{\sP}([p,z]) := \Bigl( [p(0), z(0)], \sQ(p,z) \Bigr) \in [V]^2 \times [\sV^*]^2.
  \end{equation}
  Then, we can check that $\overline{\sP}$ is the inverse map of $\sP$. Moreover, noting that:
  \begin{equation}
    |[p(0), z(0)]|_{[V]^2} \leq |[p,z]|_{C([0,T];[V]^2)} \leq \|[p,z]\|. \label{P_bdd_2}
  \end{equation}
  we can say that $\overline{\sP}$ is a bounded linear operator from $\sZ$ to $[V]^2 \times [\sV^*]^2$, asserted in \eqref{sP_bdd}, with the constant $M_0^*$, given by:
  \begin{align}
    &M_0^* = 2^{-\frac{1}{2}}(\tilde C_4 + 1)^{-1} \cdot
    \\
    &\cdot (|a|_{C([0,T];H)} + |b|_\sH + |c|_\sH + |\lambda|_\sH + |\xi|_\sH + |A|_{[L^\infty(Q)]^{N \times N}} + |\omega|_{[L^\infty(Q)]^N} + 1)^{-1}.
  \end{align}

  Thus, we finish the proof of Corollary \ref{cor_bddness}.

\end{proof}

\section{Proofs of Main Theorems}

This section is devoted to the proofs of Main Theorems. The each proof will be presented in individual subsections.

\subsection{Proof of Main Theorem \ref{mth1}} \label{proof1}
Fix $\varepsilon \geq 0$. Also, we fix a forcing pair $[\tilde u, \tilde v] \in \sU_\ad$. Invoking the definition of $\sJ_\varepsilon$ as in \eqref{cost}, we have the following estimate:
\begin{equation}
  0 \leq \underline{\sJ_\varepsilon} := \inf_{[u,v] \in \sU_\ad} \sJ_\varepsilon(u,v) \leq \sJ_\varepsilon(\tilde u, \tilde v) < \infty.
\end{equation}
Therefore, we can find a sequence of forcing pairs $\{ [u_n, v_n] \}_{n=1}^\infty \subset \sU_\ad$ such that
\begin{equation}
  \sJ_\varepsilon(u_n, v_n) \downarrow \underline{\sJ_\varepsilon} \mbox{ as } n \to \infty, \label{pr:exist1}
\end{equation}
and
\begin{equation}
  \frac{M_u}{2} |u_n|_{\sH}^2 + \frac{M_v}{2} |v_n|_\sH^2 \leq \sJ_\varepsilon(\tilde u, \tilde v) < \infty. \label{pr:exist2}
\end{equation}
By using the estimate \eqref{pr:exist2}, we see that there exist a subsequence $\{ [u_n, v_n] \}_{n=1}^\infty \subset \sU_\ad$ (not relabeled), and a pair of functions $[u^*, v^*] \subset \sU_\ad$ such that
\begin{equation}
  [\sqrt{M_u} u_n, \sqrt{M_v} v_n] \to [\sqrt{M_u} u^*, \sqrt{M_v} v^*] \mbox{ weakly in } [\sH]^2 \mbox{ as } n \to \infty. \label{pr:exist3}
\end{equation}

Now, based on Key-Lemma \ref{lem:solvability}, we let $[\eta^*, \theta^*]$ be the solution to (S)$_\varepsilon$ for the initial data $[\eta_0, \theta_0]$ and the forcings $[u^*, v^*]$, and let $[\eta_n, \theta_n]$ be the solution to (S)$_\varepsilon$ corresponding to the initial data $[\eta_0, \theta_0]$ and forcing pair $[u_n, v_n]$, for any $n \in \N$. Noting that
\begin{equation}
  \left\{ \begin{aligned}
    &[\eta_n(0), \theta_n(0)] = [\eta^*(0), \theta^*(0)] = [\eta_0, \theta_0] \mbox{ in } H,
    \\
    &|u_n|_{L^\infty(Q)} \leq |\underline{u}|_{L^\infty(Q)} \vee |\overline{u}|_{L^\infty(Q)},
  \end{aligned}
  \right. \ \mbox{ for any } n \in \N, 
\end{equation}
we can derive the following convergences as $n \to \infty$ by virtue of Key-Lemma \ref{lem:CD}:
\begin{equation}
  [\eta_n, \theta_n] \to [\eta^*, \theta^*] \mbox{ in } [C([0,T];H)]^2 \mbox{ as } n \to \infty. \label{pr:exist4}
\end{equation}
Therefore, on account of \eqref{pr:exist3} and \eqref{pr:exist4}, we can compute that:
\begin{align}
  \sJ_\varepsilon(u^*, v^*) &= \frac{M_\eta}{2}|\eta^* - \eta_\ad|_\sH^2 + \frac{M_\theta}{2} |\theta^* - \theta_\ad|_\sH^2 + \frac{M_u}{2} |u^*|_\sH^2 + \frac{M_v}{2} |v^*|_\sH^2 
  \\
  &\leq \frac{M_\eta}{2}\lim_{n \to \infty} |\eta_n - \eta_\ad|_\sH^2 + \frac{M_\theta}{2} \lim_{n \to \infty}|\theta_n - \theta_\ad|_\sH^2 \label{pr:exist6}
  \\
  &\qquad + \frac{1}{2} \varliminf_{n \to \infty} |\sqrt{M_u} u_n|_\sH^2 + \frac{1}{2} \varliminf_{n \to \infty} |\sqrt{M_v} v_n|_\sH^2
  \\
  &\leq \varliminf_{n \to \infty} \sJ_\varepsilon(u_n, v_n) = \underline{\sJ_\varepsilon}.
\end{align}
\eqref{pr:exist6} implies that 
\begin{equation}
  \sJ_\varepsilon(u^*, v^*) = \min_{[u,v] \in \sU_\ad} \sJ_\varepsilon(u,v),
\end{equation}
and thus, we complete the proof of Main Theorem \ref{mth1}. \qed

\subsection{Proof of Main Theorem \ref{mth2}} \label{proof2}
Let $[\tilde u, \tilde v] \in \sU_\ad$ be fixed. Let $[\tilde \eta, \tilde \theta]$ be the solution to (S)$_\varepsilon$ for the initial pair $[\eta_0, \theta_0]$ and forcings $[\tilde u, \tilde v]$, and for any $n \in \N$, $[\tilde \eta_n, \tilde \theta_n]$ be the solution to (S)$_{\varepsilon_n}$ for initial pair $[\eta_{0,n}, \theta_{0,n}]$ and forcing pair $[\tilde u, \tilde v]$. Applying Key-Lemma \ref{lem:CD}, we can see that
\begin{equation}
  [\tilde \eta_n, \tilde \theta_n] \to [\tilde \eta, \tilde \theta] \mbox{ in } C([0,T];H) \mbox{ as } n \to \infty, \mbox{ and } \lim_{n \to \infty} \sJ_{\varepsilon_n} (\tilde u, \tilde v) = \sJ_\varepsilon(\tilde u, \tilde v), \label{pr:CD1}
\end{equation} 
and hence,
\begin{equation}
  \overline{\sJ} := \sup_{n \in \N} \sJ_{\varepsilon_n} (\tilde u, \tilde v) < \infty.
\end{equation}

Next, since $[u_n^*, v_n^*]$ is an optimal control of (OCP)$_{\varepsilon_n}$, it is computed that
\begin{equation}
  \frac{M_u}{2}|u_n^*|_\sH^2 + \frac{M_v}{2}|v_n^*|_\sH^2 \leq \sJ_{\varepsilon_n} (u_n^*, v_n^*) \leq \sJ_{\varepsilon_n}(\tilde u, \tilde v) \leq \overline{\sJ}, \mbox{ for } n = 1,2,3,\dots.
\end{equation}
Therefore, there exist a subsequence $\{ n_j \}_{j=1}^\infty \subset \{ n \}$ and a pair of function $[u^*, v^*] \in [\sH]^2$ such that 
\begin{equation}
  [\sqrt{M_u} u_n^*, \sqrt{M_v} v_n^*] \to [\sqrt{M_u} u^*,\sqrt{M_v} v^*] \mbox{ weakly in } [\sH]^2 \mbox{ as } n \to \infty. \label{pr:CD2}
\end{equation}
Moreover, invoking:
\begin{equation}
  \sup_{n \in \N} |u_n^*|_{L^\infty(Q)} \leq |\underline{u}|_{L^\infty(Q)} \vee |\overline{u}|_{L^\infty(Q)},
\end{equation}
we can apply Key-Lemma \ref{lem:CD} and obtain that
\begin{equation}
  [\eta_n^*, \theta_n^*] \to [\eta^*, \theta^*] \mbox{ in } C([0,T];H) \mbox{ as } n \to \infty. \label{pr:CD3}
\end{equation}

Now, having in mind \eqref{pr:CD1}--\eqref{pr:CD3}, we arrive at:\noeqref{pr:CD2}
\begin{align}
  \sJ_\varepsilon(u^*, v^*) &= \frac{M_\eta}{2}|\eta^* - \eta_\ad|_\sH^2 + \frac{M_\theta}{2} |\theta^* - \theta_\ad|_\sH^2 + \frac{M_u}{2} |u^*|_\sH^2 + \frac{M_v}{2} |v^*|_\sH^2 
  \\
  &\leq \frac{M_\eta}{2}\lim_{n \to \infty} |\eta_n^* - \eta_\ad|_\sH^2 + \frac{M_\theta}{2} \lim_{n \to \infty}|\theta_n^* - \theta_\ad|_\sH^2 \label{pr:CD4}
  \\
  &\qquad + \frac{1}{2} \varliminf_{n \to \infty} |\sqrt{M_u} u_n^*|_\sH^2 + \frac{1}{2} \varliminf_{n \to \infty} |\sqrt{M_v} v_n^*|_\sH^2
  \\
  &\leq \varliminf_{n \to \infty} \sJ_{\varepsilon_n}(u_n^*, v_n^*) \leq \lim_{n \to \infty} \sJ_{\varepsilon_n} (\tilde u, \tilde v) = \sJ_\varepsilon (\tilde u, \tilde v).
\end{align}
Since the pair $[\tilde u, \tilde v] \in \sU_\ad$ is arbitrary, \eqref{pr:CD4} means that $[u^*, v^*]$ is an optimal control of (OCP)$_\varepsilon$.

Thus, the proof of Main Theorem \ref{mth2} is completed. \qed

\subsection{Proof of Main Theorem \ref{mth3}} \label{proof3}

\begin{keylem} \label{lem:Gateaux}
  Set $\sX := L^\infty(Q) \times \sH$. Let $\varepsilon \in (0,1)$ be fixed, and let $\Lambda_\varepsilon:\sX \longrightarrow [\sH]^2$ be the solution operator to (S)$_\varepsilon$ which maps $[u,v] \in \sX$ to the solution $\Lambda_\varepsilon(u,v) := [\eta, \theta]$ to the state system (S)$_\varepsilon$. 
  
  \smallskip
  \noindent
  \textbf{(I)} The restriction of solution operator $\Lambda_\varepsilon|_\sX:\sX \longrightarrow [\sH]^2$ is G\^{a}teaux differentiable over $\sX$. More precisely,
  \begin{gather}
    \bigl( (\Lambda_\varepsilon|_\sX)'(u,v) \bigr)(h,k) = D_{[h,k]} \Lambda_\varepsilon(u,v) := \lim_{\delta \to 0} \frac{\Lambda_\varepsilon(u + \delta h, v + \delta k) - \Lambda_\varepsilon(u,v)}{\delta}
    \\
    = \sP_\varepsilon(L_u h,L_v k), \mbox{ for any } [u,v] \in \sX, \mbox{ and any direction } [h,k] \in \sX,
  \end{gather}
  In this context, $\sP_\varepsilon \in \sL([\sH]^2)$ is given as a restriction $\sP|_{\{[0,0]\} \times [\sH]^2}$ of the bounded linear operator $\sP = \sP(a,b,c,\lambda,\xi,\omega,A)$, defined in \eqref{sol_op_P} in the case:
  \begin{equation}
    \left\{ \begin{aligned}
      &a = \alpha_0(\eta),
      \\
      &b = 0,
      \\
      &c = \alpha_0'(\eta) \partial_t \theta,
      \\
      &\lambda = g'(\eta) + \alpha''(\eta) \gamma_\varepsilon(\nabla \theta),
      \\
      &\xi = 0,
      \\
      &\omega = \alpha'(\eta) \nabla \gamma_\varepsilon(\nabla \theta),
      \\
      &A = \alpha(\eta) \nabla^2 \gamma_\varepsilon(\nabla \theta),
    \end{aligned} \right. \mbox{ in } [\sH]^7.
  \end{equation}
  Moreover, for any $(u,v) \in \sX$, the G\^{a}teaux derivative $(\Lambda_\varepsilon|_\sX)'(u,v)$ can be uniquely extended to $[\sH]^2$.
  
  \noindent
  \textbf{(II)} The restriction of cost functional $\sJ_\varepsilon|_\sX:\sX \longrightarrow [0,\infty)$ is G\^{a}teaux differentiable over $\sX$. Furthermore, for any $[u,v] \in \sX$, the G\^{a}teaux derivative $(\sJ_\varepsilon|_\sX)'(u,v)$ admits a unique extension $\sJ_\varepsilon'(u,v) \in ([\sH]^2)^* = [\sH]^2$ such that
  \begin{gather}
    \begin{aligned}
      &\bigl(\sJ_\varepsilon'(u,v), [h,k] \bigr)_{[\sH]^2} = D_{[h,k]} \sJ_\varepsilon(u,v) := \lim_{\delta \to 0} \frac{\sJ_\varepsilon(u + \delta h, v + \delta k) - \sJ_\varepsilon(u,v)}{\delta}
      \\
      &= ([M_\eta(\eta - \eta_{\rm ad}), M_\theta(\theta - \theta_{\rm ad})], \sP_\varepsilon(L_u h,L_v k))_{[\sH]^2} + ([M_u u,M_v v], [h,k])_{[\sH]^2},
    \end{aligned}
    \\
    \mbox{ for any } [u,v] \in \sX, \mbox{ and any direction } [h,k] \in \sX.
  \end{gather}
  
\end{keylem}
\begin{proof}
  Let $[u,v], \,[h,k] \in \sX$ be fixed. Let us take any constant $\delta \in [-1,1] \setminus \{ 0 \}$, and set $[\eta^\delta, \theta^\delta] := \Lambda_\varepsilon(u + \delta h, v + \delta k)$. Since it holds that 
  \begin{equation}
    [u + \delta h, v + \delta k] \to [u,v] \mbox{ in } \sX \mbox{ as } \delta \to 0,
  \end{equation}
  Key-Lemma \ref{lem:CD} implies the following convergences as $\delta \to 0$:
  \begin{gather}
    \left\{ \begin{aligned}
      &\eta^\delta \to \eta \mbox{ in } C([0,T];H), \, \sV, \ \mbox{ weakly in } W^{1,2}(0,T;V), \mbox{ and weakly-$*$ in } L^\infty(Q),
      \\
      &\theta^\delta \to \theta \mbox{ in } C([0,T];H), \, \sV, \ \mbox{ weakly in } W^{1,2}(0,T;V),
      \\
      &\eta^\delta(t) \to \eta(t), \, \theta^\delta(t) \to \theta(t) \mbox{ in } V, \ \mbox{ for any } t \in [0,T].
    \end{aligned} \right. \label{Gateaux00}
  \end{gather}
  Additionally, we can assume that:
  \begin{equation}
    \eta^\delta \to \eta, \, \theta^\delta \to \theta, \, \nabla \theta^\delta \to \nabla \theta \mbox{ in the pointwise sense, a.e. in } Q \mbox{ as } \delta \to 0. \label{Gateaux00-1}
  \end{equation}

  Now, we set $[\chi^\delta, \zeta^\delta] = [\frac{\eta^\delta - \eta}{\delta}, \frac{\theta^\delta - \theta}{\delta}]$. Then, based on standard linearization process (cf. \cite{MR4395725}), we can say that $[\chi^\delta, \zeta^\delta]$ solves the following variational system:
  \begin{gather}
    (\partial_t \chi^\delta(t) + \omega^\delta(t) \cdot \nabla \zeta^\delta(t), \varphi)_H + \int_\Omega \lambda^\delta(t) \chi^\delta(t) \varphi \,dx +(\nabla \chi^\delta(t), \nabla \varphi)_{[H]^N} \label{Gateaux01}
    \\
    + \mu^2 (\nabla \partial_t \chi^\delta(t), \nabla \varphi)_{[H]^N} = (L_u h(t),\varphi)_H, \mbox{ for any } \varphi \in V, \mbox{ and for a.e. } t \in (0,T),
  \end{gather}
  and
  \begin{gather}
    \int_\Omega (a^\delta(t) \partial_t \zeta^\delta(t) + c^\delta(t) \chi^\delta(t)) \psi \,dx + (A^\delta(t) \nabla \zeta^\delta(t) + \chi^\delta(t)\omega^\delta(t), \nabla \psi)_{[H]^N} \label{Gateaux02}
    \\
    + \nu^2(\nabla \partial_t \zeta^\delta(t), \nabla \psi)_{[H]^N} = \langle \overline{k}^\delta(t),\psi \rangle_V, \ \mbox{ for any }  \psi \in V, \ \mbox{ and for a.e. } t \in (0,T),
  \end{gather}
  where
  \begin{equation}
    \left\{ \begin{aligned}
      &\overline{a}^\delta := \alpha_0(\eta^\delta),
      \\
      &\overline{c}^\delta := \partial_t \theta \int_0^1 \alpha_0'(\eta + \delta \sigma \chi^\delta)\,d\sigma,
      \\
      &\overline{\lambda}^\delta := \int_0^1 \bigl(g'(\eta + \delta \sigma \chi^\delta) + \alpha''(\eta + \delta \sigma \chi^\delta) \gamma_\varepsilon(\nabla \theta^\delta)\bigr) \,d\sigma
      \\
      &\overline{\omega}^\delta := \alpha'(\eta) \int_0^1 \nabla \gamma_\varepsilon(\nabla (\theta + \delta \sigma \zeta^\delta)) \,d\sigma,
      \\
      &\overline{A}^\delta := \alpha(\eta) \int_0^1 \nabla^2 \gamma_\varepsilon(\nabla (\theta + \delta \sigma \zeta^\delta)) \,d\sigma,
    \end{aligned} \right. \mbox{ in } [\sH]^5,
  \end{equation}
  and
  \begin{equation}
    \overline{k}^\delta := L_v k - \diver\biggl( \chi^\delta \overline{\omega}^\delta - \chi^\delta \int_0^1 \alpha'(\eta + \delta \sigma \chi^\delta) \nabla \gamma_\varepsilon(\nabla \theta^\delta)\,d\sigma \biggr) \mbox{ in } \sV^*.
  \end{equation}
  In other words, by setting an operator $\sP_\delta$ as the restriction $\sP|_{[0,0] \times [\sV]^2}$ of the linear isomorphism $\sP = \sP(a,b,c,\lambda,\xi,\omega,A)$ defined in \eqref{sol_op_P} in the case:
  \begin{equation}
    [a,b,c,\lambda,\xi,\omega,A] := [\overline{a}^\delta, 0, \overline{c}^\delta, \overline{\lambda}^\delta, 0, \overline{\omega}^\delta, \overline{A}^\delta] \mbox{ in } [\sH]^7,
  \end{equation}
  we can see that $[\chi^\delta, \zeta^\delta] = \sP_\delta(L_u h,\overline{k}^\delta)$. 
  
  Now, noting that:
  \begin{gather}
    [\delta \chi^\delta, \delta \zeta^\delta] = [\eta^\delta - \eta, \theta^\delta - \theta] \to [0,0], \mbox{ in } [L^\infty(0,T;V)]^2, \mbox{ and }
    \\
    \delta \chi^\delta \to 0 \mbox{ weakly-$*$ in } L^\infty(Q),
  \end{gather}
  using the convergences \eqref{Gateaux00}, \eqref{Gateaux00-1} and Lebesgue's dominated convergence theorem yields the following convergences as $\delta \to 0$:
  \begin{gather}
    \overline{a}^\delta \to a \mbox{ in } \sH \mbox{ and weakly in } W^{1,2}(0,T;H),
    \\
    [\overline{c}^\delta, \overline{\lambda}^\delta] \to [c, \lambda] \mbox{ in } [\sH]^2, \label{Gateaux03}
  \end{gather}
  and
  \begin{equation}
    \overline{\omega}^\delta \to \omega, \,\overline{A}^\delta \to A \mbox{ weakly-$*$ in } L^\infty(Q), \mbox{ and in the pointwise sense, a.e. in } Q. \label{Gateaux04}
  \end{equation}\noeqref{Gateaux04}
  Now, to apply Corollary \ref{cor:CD}, we derive the convergence of $\overline{k}^\delta$. By setting:
  \begin{align}
    m_* := \sup_{0 < |\delta| < 1} |\eta^\delta|_{L^\infty(Q)} \geq |\eta|_{L^\infty(Q)},
  \end{align}
  one can see that
  \begin{align}
    |\overline{k}^\delta(t)|_{V^*}^2 &\leq \Bigl( L_v |k(t)|_{V^*} + |\chi^\delta(t)|_H \Bigl(|\alpha'(\eta(t))|_{L^\infty(\Omega)} + \int_0^1 |\alpha'(\eta + \delta \sigma \chi^\delta)(t)|_{L^\infty(\Omega)}\,d\sigma \Bigr) \Bigr)^2
    \\
    &\leq 8 \bigl( L_v^2 |k(t)|_{V^*}^2 + |\chi^\delta(t)|_H^2 |\alpha'|_{L^\infty(-m_*, m_*)}^2 \bigr), 
    \\
    &\qquad \mbox{ for a.e. } t \in (0,T), \mbox{ and for any } \delta \in [-1,1] \setminus \{ 0 \}.
  \end{align}
  Now, we set:
  \begin{align}
    &\begin{gathered}
      J_\delta(t) := |\chi^\delta(t)|_H^2 + \mu^2 |\nabla \chi^\delta(t)|_{[H]^N}^2 + |\sqrt{\overline{a}^\delta(t)} \zeta^\delta(t)|_H^2 + \nu^2 |\nabla \zeta^\delta(t)|_{[H]^N}^2,
      \\
      \mbox{ for any } t \in [0,T], \mbox{ and }
    \end{gathered}
    \\
    &\begin{aligned}
      R^{*}(t) &:= \left( \frac{12((C_V^{L^4})^2 + 1)}{1 \wedge \mu^2 \wedge \delta_* \wedge \nu^2} \right) \cdot
      \\
      &\qquad \cdot \sup_{\delta \in (0,1)} \bigl(|\partial_t \overline{a}^\delta(t)|_H + |\overline{c}^\delta(t)|_H + |\overline{\lambda}^\delta(t)|_H + |\overline{\omega}^\delta(t)|_{[L^\infty(\Omega)]^N} + 1 \bigr).
    \end{aligned}
  \end{align}
  We apply Key-Lemma \ref{lem:P_uni} under the case when:
  \begin{gather}
    \left\{ \begin{aligned}
      &[a^1, b^1, c^1, \lambda^1, \xi^1, \omega^1, A^1] = [a^2, b^2, c^2, \lambda^2, \xi^2, \omega^2, A^2] 
      \\
      &\qquad = [\overline{a}^\delta, 0, \overline{c}^\delta, \overline{\lambda}^\delta, 0, \overline{\omega}^\delta, \overline{A}^\delta] \mbox{ in } [\sH]^7,
      \\
      &[p_0^1, z_0^1] = [p_0^2, z_0^2] = [0,0] \mbox{ in } [V]^2,
      \\
      &[h^1, k^1] = [L_u h, \overline{k}^\delta], \, [h^2, k^2] = [0,0] \mbox{ in } [\sV^*]^2,
      \\
      &[p^1,z^1] = \sP_\delta(L_u h, \overline{k}^\delta), \, [p^2, z^2] = [0,0] \mbox{ in } [\sH]^2,
    \end{aligned} \right.
    \\
    \mbox{ for any } \delta \in [-1,1] \setminus \{ 0 \},
  \end{gather}
  and find the following estimate:
  \begin{align}
    \frac{d}{dt}J_\delta(t) &\leq R^*(t) J_\delta(t) + L_u^2 |h(t)|_{V^*}^2 + |\overline{k}^\delta(t)|_{V^*}^2
    \\
    &\leq (R^*(t) + 8|\alpha'|_{L^\infty(I_*)}^2) J_\delta(t) + 8|[L_u h, L_v k](t)|_{V^*}^2,
    \\
    &\mbox{ for a.e. } t \in (0,T), \mbox{ and any } \delta \in [-1,1] \setminus \{ 0 \}.
  \end{align}
  Subsequently, by using Gronwall's lemma, we can derive that:
  \begin{itemize}
    \item[($\sharp 0$)] $\{ \chi^\delta \}_{\delta \in (0,1)}$ and $\{ \zeta^\delta \}_{\delta \in (0,1)}$ are bounded in $L^\infty(0,T;V)$.
  \end{itemize}
  Hence, combining \eqref{Gateaux01}--\eqref{Gateaux02} and ($\sharp 1$), we can compute that
  \begin{align}
    \langle \overline{k}^\delta - L_v k, \psi \rangle_{\sV} &= \left( \chi^\delta, \biggl( \overline{\omega}^\delta - \int_0^1 \alpha'(\eta + \delta \sigma \chi^\delta) \nabla \gamma_\varepsilon(\nabla \theta^\delta) \,d\sigma \biggr) \cdot \nabla \psi \right)_\sH
    \\
    &\to 0, \ \mbox{ for any } \psi \in \sV,
  \end{align}
  and therefore,
  \begin{equation}
    \overline{k}^\delta \to L_v k \mbox{ weakly in } \sV^* \mbox{ as } \delta \to 0. \label{Gateaux05}
  \end{equation}
  
  On account of the convergences \eqref{Gateaux03}--\eqref{Gateaux05}, we can apply Corollary \ref{cor:CD} and obtain that:
  \begin{gather}
    [\chi^\delta, \zeta^\delta] = \sP_\delta(L_u h, \overline{k}^\delta) \to [\chi, \zeta] = \sP_\varepsilon(L_u h, L_v k) \mbox{ in } C([0,T];H)
    \\
    \mbox{ and weakly in } W^{1,2}(0,T;V), \mbox{ as } \delta \to 0. \label{Gateaux06}
  \end{gather}
  This convergence and Corollary \ref{cor_bddness} imply that:
  \begin{itemize}
    \item[($\sharp 1$)] $\Lambda_\varepsilon|_\sX$ is G\^{a}teaux differentiable over $\sX$, and bounded in the norm of $[\sH]^2$.
  \end{itemize}
  Therefore, $(\Lambda_\varepsilon|_\sX)'$ admits a continuous linear extension to $[\sH]^2$, and thus, we complete the proof of the assertion (I).

  Next, we verify the statement of (II). Let $[u,v] \in \sX$ be fixed. One can compute that
  \begin{gather}
    \begin{aligned}
      &\frac{\sJ_\varepsilon(u + \delta h, v + \delta k) - \sJ_\varepsilon(u,v)}{\delta}
      \\
      &\quad = M_\eta \biggl( \Bigl( \frac{\eta^\delta + \eta}{2} - \eta_\ad \Bigr), \chi^\delta \biggr)_\sH + M_\theta \biggl( \Bigl( \frac{\theta^\delta + \theta}{2} - \theta_\ad \Bigr), \zeta^\delta \biggr)_\sH \label{Gateaux07}
      \\
      &\qquad + M_u \biggl( \Bigl( u + \frac{\delta h}{2} \Bigr), h \biggr)_\sH + M_v \biggl( \Bigl( v + \frac{\delta k}{2} \Bigr), k \biggr)_\sH,
    \end{aligned}
    \\
    \mbox{ for any direction } [h,k] \in \sX.
  \end{gather}
  From the convergences \eqref{Gateaux00} and \eqref{Gateaux06}, letting $\delta \to 0$ in \eqref{Gateaux07} yields that
  \begin{gather}
    D_{[h,k]} \sJ_\varepsilon(u,v) = ([M_\eta(\eta - \eta_{\rm ad}), M_\theta(\theta - \theta_{\rm ad})], \sP_\varepsilon(L_u h,L_v k))_{[\sH]^2} 
    \\
    + ([M_u u,M_v v], [h,k])_{[\sH]^2}, \mbox{ for any } [h,k] \in \sX.
  \end{gather}
  Now, having in mind the estimates in Corollary \ref{cor_bddness}, the following inequality holds:
  \begin{align}
    &|D_{[h,k]} \sJ_\varepsilon(u,v)| \leq \bigl( M_*^1(M_\eta + M_\theta)(L_u + L_v) + M_u + M_v \bigr) \cdot 
    \\
    &\quad \cdot \bigl( |[\eta - \eta_\ad, \theta - \theta_\ad]|_{[\sH]^2} + |[u,v]|_{[\sH]^2} \bigr) |[h,k]|_{[\sH]^2}, \mbox{ for any } [h,k] \in \sX.
  \end{align}
  Based on the above, one can say that:
  \begin{itemize}
    \item[($\sharp 2$)] the mapping $[h,k] \in \sX \mapsto D_{[h,k]} \sJ_\varepsilon(u,v)$ is a linear functional, and bounded in the norm of $[\sH]^2$.
  \end{itemize}
  Therefore, since $\sX$ is a dense subset of $[\sH]^2$, Riesz's theorem enable us to find the required functional $\sJ_\varepsilon'(u,v) \in ([\sH]^2)^* = [\sH]^2$, which is the unique extension of the G\^{a}teaux derivative of $\sJ_\varepsilon|_\sX(u,v) \in \sX^*$ at $[u,v] \in \sX$.

  Thus, we finish the proof of Key-Lemma \ref{lem:Gateaux}.

\end{proof}

\begin{keylem} \label{lem:conj}
  Let $\varepsilon \in (0,1)$ be fixed. Let $[u_\varepsilon^*, v_\varepsilon^*]$ be the optimal control of the problem (OCP)$_\varepsilon$ and let $[\eta_\varepsilon, \theta_\varepsilon]$ be the solution to the state system (S)$_\varepsilon$ with the initial data $[\eta_0, \theta_0]$ and the forcing pair $[u^*,v^*]$. We define an operator $\overline{\sP}_\varepsilon \in \sL([\sH]^2)$ by the restriction $\sP|_{\{ [0,0] \} \times [\sH]^2}$ of the bounded linear operator $\sP = \sP(a,b,c,\lambda,\xi,\omega,A)$ as in \eqref{sol_op_P} when:
  \begin{equation}
    \left\{ \begin{aligned}
      &a = \alpha_0(\eta_\varepsilon),
      \\
      &b = 0,
      \\
      &c = \alpha_0'(\eta_\varepsilon) \partial_t \theta_\varepsilon,
      \\
      &\lambda = g'(\eta_\varepsilon) + \alpha''(\eta_\varepsilon) \gamma_\varepsilon(\nabla \theta_\varepsilon),
      \\
      &\xi = 0,
      \\
      &\omega = \alpha'(\eta_\varepsilon) \nabla \gamma_\varepsilon(\nabla \theta_\varepsilon),
      \\
      &A = \alpha(\eta_\varepsilon) \nabla^2 \gamma_\varepsilon(\nabla \theta_\varepsilon),
    \end{aligned} \right. \mbox{ in } [\sH]^7.
  \end{equation}
  Then, the operators $\sP_\varepsilon^*$, defined in Remark \ref{rem:p*}, and $\sP_\varepsilon$ have a conjugate relationship, in the following sense:
  \begin{equation}
    (\sP_\varepsilon^*(u,v),[h,k])_{[\sH]^2} = ([u,v], \overline{\sP}_\varepsilon(h,k))_{[\sH]^2}, \mbox{ for any } [u,v], [h,k] \in [\sH]^2.
  \end{equation}
\end{keylem}
\begin{proof}
  Let the pair of functions $[u,v], [h,k] \in [\sH]^2$ be fixed. Also, we denote $[p,z] := \sP_\varepsilon^*(u,v)$ and $[\chi,\zeta] := \overline{\sP}_\varepsilon(h,k)$. Then, using integration by parts, one can be computed that
  \begin{align}
    &\quad (\sP_\varepsilon^*(u,v), [h,k])_{[\sH]^2} = \int_0^T (h(t),p(t))_H \,dt + \int_0^T (k(t), z(t))_H \,dt
    \\
    &= \int_0^T \Biggl[\Bigl(\bigl(\partial_t \chi + (g'(\eta) + \alpha''(\eta) \gamma_\varepsilon(\nabla \theta))\chi + \alpha'(\eta)\nabla \gamma_\varepsilon(\nabla \theta) \cdot \nabla \zeta\bigr)(t), p(t)\Bigr)_H
    \\
    &\qquad + (\nabla (\chi + \mu^2 \partial_t \chi)(t), \nabla p(t))_{[H]^N} + \bigl( \bigl( \alpha_0(\eta) \partial_t \zeta + \alpha_0'(\eta) \chi \partial_t \theta \bigr)(t), z(t)\bigr)_H 
    \\
    &\qquad + \Bigl(\bigl( \alpha(\eta) \nabla^2 \gamma_\varepsilon(\nabla \theta) \nabla \zeta + \nu^2 \nabla \partial_t \zeta + \alpha'(\eta) \chi \nabla \gamma_\varepsilon(\nabla \theta) \bigr)(t), \nabla z(t)\Bigr)_{[H]^N} \Biggr]\,dt
    \\
    &= \Bigl[ (p,\chi)_H + \mu^2(\nabla p, \nabla \chi)_{[H]^N} + (\alpha_0(\eta)z,\zeta)_H + \nu^2 (\nabla z, \nabla \zeta)_{[H]^N} \Bigr](T)
    \\
    &\quad - \Bigl[ (p,\chi)_H + \mu^2(\nabla p, \nabla \chi)_{[H]^N} + (\alpha_0(\eta)z,\zeta)_H + \nu^2 (\nabla z, \nabla \zeta)_{[H]^N} \Bigr](0)
    \\
    &\quad + \int_0^T \Biggl[\Bigl(\bigl(-\partial_t p + (g'(\eta) + \alpha''(\eta) \gamma_\varepsilon(\nabla \theta))p + \alpha'(\eta) \nabla \gamma_\varepsilon(\nabla \theta) \cdot \nabla z\bigr)(t), \chi(t) \Bigr)_H 
    \\
    &\quad\qquad + (\nabla (p - \mu^2 \partial_t p)(t), \nabla \chi(t))_{[H]^N} - \bigl( \partial_t (\alpha_0(\eta) z)(t), \zeta(t) \bigr)_H
    \\
    &\quad\qquad + \Bigl(\bigl(\alpha(\eta) \nabla^2 \gamma_\varepsilon(\nabla \theta) \nabla z - \nu^2 \nabla \partial_t z + \alpha'(\eta) \nabla \gamma_\varepsilon(\nabla \theta)\bigr)(t), \nabla \zeta(t)\Bigr)_{[H]^N} \Biggr]\,dt 
    \\
    &= ([u,v], [\chi,\zeta])_{[\sH]^2} = ([u,v], \overline{\sP}_\varepsilon(h,k))_{[\sH]^2}.
  \end{align}
  This computation finishes the proof of Key-Lemma \ref{lem:conj}.
\end{proof}

\begin{proof}[The proof of Main Theorem \ref{mth3}]
  Let $[u^*,v^*] \in \sU_\ad$ be an optimal control of (OCP)$_\varepsilon$, with the solution $[\eta_\varepsilon^*, \theta_\varepsilon^*]$ to (S)$_\varepsilon$ corresponding to the initial data $[\eta_0, \theta_0]$ and forcing pair $[u^*,v^*]$. Also, we set $[p_\varepsilon^*, z_\varepsilon^*] = \sP_\varepsilon^*(M_\eta(\eta_\varepsilon^* - \eta_\ad), M_\theta(\theta_\varepsilon^* - \theta_\ad))$, as in Remark \ref{rem:p*}. Since the cost functional takes its minimum at $[u^*, v^*] \in \sU_\ad$, Key-Lemma \ref{lem:Gateaux} and \ref{lem:conj} allow us to compute that:
  \begin{align}
    0 &\leq (\sJ_\varepsilon'(u^*,v^*),[h,k])_{[\sH]^2} = \lim_{\delta \downarrow 0} \frac{\sJ_\varepsilon(u^* + \delta(h - u^*), v^* + \delta k) - \sJ_\varepsilon(u^*,v^*)}{\delta}
    \\
    &= ([M_\eta(\eta_\varepsilon^* - \eta_\ad), M_\theta(\theta_\varepsilon^* - \theta_\ad)], \overline{\sP}_\varepsilon(L_u (h - u^*), L_v k))_{[\sH]^2} 
    \\
    &\qquad + ([M_u u^*, M_v v^*], [h - u^*,k])_{[\sH]^2} \label{NC1}
    \\
    &= (\sP_\varepsilon^*(M_\eta(\eta_\varepsilon^* - \eta_\ad), M_\theta(\theta_\varepsilon^* - \theta_\ad)), [L_u (h - u^*), L_v k])_{[\sH]^2} 
    \\
    &\qquad + ([M_u u^*, M_v v^*], [h - u^*,k])_{[\sH]^2}
    \\
    &= ([L_u p_\varepsilon^* + M_u u^*, L_v z_\varepsilon^* + M_v v^*], [h - u^*, k])_{[\sH]^2}, \ \mbox{ for any } [h,k] \in \sU_\ad.
  \end{align}

  Now, letting $h = u^*$ in \eqref{NC1} yields that:  
  \begin{equation}
    (L_v z_\varepsilon^* + M_v v^*, k)_\sH \geq 0, \ \mbox{ for all } k \in \sH, \label{NC2}
  \end{equation}
  which implies \eqref{NC0b}, and applying this to \eqref{NC1}, we have \eqref{NC0a}.

  Thus, we complete the proof of Main Theorem \ref{mth3}.
\end{proof}

\subsection{Proof of Main Theorem \ref{mth4}} \label{proof4}
Let $[u_\varepsilon^*, v_\varepsilon^*]$ be an optimal control of (OCP)$_\varepsilon$ for $\varepsilon \in (0,1)$, and let $[\eta_\varepsilon^*, \theta_\varepsilon^*]$ be the solution to (S)$_\varepsilon$ corresponding to the initial data $[\eta_0, \theta_0]$ and $[u_\varepsilon^*, v_\varepsilon^*]$ as a forcing pair. By virtue of Main Theorem \ref{mth2} and Key-Lemma \ref{lem:CD}, we can find a subsequence $\{ \varepsilon_n \}_{n=1}^\infty \subset \{ \varepsilon \}$ with $\varepsilon_n \downarrow 0$ as $n \to \infty$, and find limiting pairs $[u^\circ, v^\circ] \in \sU_\ad$, $[\eta^\circ, \theta^\circ] \in [\sH]^2$, and a vector-valued function $\varpi^\circ \in [L^\infty(Q)]^N$, such that:
\begin{equation}
  \begin{gathered}
    \sqrt{M_u} u_{\varepsilon_n}^* \to \sqrt{M_u} u^\circ \mbox{ weakly in } \sH, \mbox{ and weakly-$*$ in } L^\infty(Q),
    \\
    \sqrt{M_v} v_{\varepsilon_n}^* \to \sqrt{M_v} v^\circ \mbox{ weakly in } \sH,
  \end{gathered} \mbox{ as } n \to \infty. \label{NC0_01}
\end{equation}
\begin{equation}
  \left\{ \begin{aligned}
    &\eta_n^* := \eta_{\varepsilon_n}^* \to \eta^\circ \mbox{ in } C([0,T];H), \, \sV, \ \mbox{ weakly in } W^{1,2}(0,T;V),
    \\
    &\qquad \mbox{ and weakly-$*$ in } L^\infty(Q),
    \\
    &\theta_n^* := \theta_{\varepsilon_n}^* \to \theta^\circ \mbox{ in } C([0,T];H), \, \sV, \ \mbox{ weakly in } W^{1,2}(0,T;V),
    \\
    &\eta_n^*(t) \to \eta^\circ(t), \, \theta_n^*(t) \to \theta^\circ(t) \mbox{ in } V, \ \mbox{ for any } t \in [0,T], \mbox{ and }
    \\
    &\eta_n^* \to \eta^\circ, \, \theta_n^* \to \theta^\circ \mbox{ in the pointwise sense, a.e. in } Q.
  \end{aligned} \right. \label{NC0_02}
\end{equation}
\begin{equation}
  \left\{ \begin{aligned}
    &a_n^* := \alpha_0(\eta_n^*) \to a^\circ := \alpha_0(\eta^\circ) \mbox{ weakly in } W^{1,2}(0,T;H), \mbox{ weakly-$*$ in } L^\infty(Q),
  \\
  &\qquad \mbox{ and in the pointwise sense, a.e. in } Q,
  \\
  &\lambda_n^* := g'(\eta_n^*) + \alpha''(\eta_n^*) \gamma_{\varepsilon_n} (\nabla \theta_n^*) \to \lambda^\circ := g'(\eta^\circ) + \alpha''(\eta^\circ) |\nabla \theta^\circ| \mbox{ in } \sH,
  \end{aligned} \right. \label{NC0_03}
\end{equation}
\begin{equation}
  \left\{ \begin{aligned}
    &\nabla \gamma_{\varepsilon_n}(\nabla \theta_n^*) \to \varpi^\circ \mbox{ weakly-$*$ in } [L^\infty(Q)]^N,
    \\
    &\omega_n^* := \alpha'(\eta_n^*) \nabla \gamma_{\varepsilon_n}(\nabla \theta_n^*) \to \omega^\circ := \alpha'(\eta^\circ) \varpi^\circ \mbox{ weakly-$*$ in } [L^\infty(Q)]^N,
  \end{aligned} \right. \mbox{ as } n \to \infty. \label{NC0_04}
\end{equation}
Due to the theory of Mosco convergence as in Remark \ref{Rem.MG} (Fact 1), one can say that
\begin{equation}
  \varpi^\circ \in \partial \gamma_0(\nabla \theta^\circ) = \Sgn(\nabla \theta^\circ), \mbox{ a.e. in } Q. \label{NC0_05}
\end{equation}

Here, we set $[p_\varepsilon, z_\varepsilon] := \sT[p_\varepsilon^*, z_\varepsilon^*] = \sP_\varepsilon^\circ(M_\eta(\eta_\varepsilon^* - \eta_\ad), M_\theta(\theta_\varepsilon^* - \theta_\ad))$, where $[p_\varepsilon^*, z_\varepsilon^*]$ is as in Main Theorem \ref{mth3}, and $\sP_\varepsilon^\circ$ is given in Remark \ref{rem:p*}. Then, by noting the convergences \eqref{NC0_02}--\eqref{NC0_05} and the following estimate:
\begin{equation}
  \sup_{\varepsilon \in (0,1)} |\alpha_0'(\eta_\varepsilon^*) \partial_t \theta_\varepsilon^*|_\sH \leq \sup_{\varepsilon \in (0,1)} \bigl( |\alpha'(\eta_\varepsilon^*)|_{L^\infty(Q)} |\partial_t \theta_\varepsilon^*|_\sH \bigr) < \infty,
\end{equation}
applying Corollary \ref{cor:P_est} yields that:
\begin{itemize}
  \item $\{ p_\varepsilon \}_{\varepsilon \in (0,1)}$ is bounded in $W^{1,2}(0,T;V)$,
  \item $\{ z_\varepsilon \}_{\varepsilon \in (0,1)}$ is bounded in $L^\infty(0,T;V)$.
\end{itemize}
Moreover, for any $\psi \in \sW$, we can compute that:
\begin{align}
  &\quad |\langle -\diver (\alpha(\eta_\varepsilon^*) \nabla^2 \gamma_\varepsilon(\nabla \theta_\varepsilon^*) \nabla z_\varepsilon^*), \psi \rangle_\sW| = |\bigl( \alpha(\eta_\varepsilon^*) \nabla^2 \gamma_\varepsilon(\nabla \theta_\varepsilon^*) \nabla z_\varepsilon^*, \nabla \psi \bigr)_{[H]^N}
  \\
  &\leq |(\alpha_0(\eta_\varepsilon^*) z_\varepsilon, \partial_t \psi)_\sH| + \nu^2 |(\nabla z_\varepsilon^*, \nabla \partial_t \psi)_{[\sH]^N}| + |(\alpha'(\eta_\varepsilon^*) p_\varepsilon^* \nabla \gamma_\varepsilon(\nabla \theta_\varepsilon^*), \nabla \psi)_{[\sH]^N}|
  \\
  &\qquad + |(M_\theta(\theta_\varepsilon^* - \theta_\ad), \psi)_\sH|
  \\
  &\leq C_5^* |\psi|_\sW,
\end{align}
with
\begin{equation}
  C_5^* := \sup_{\varepsilon \in (0,1)} \bigl( \sqrt{2}(|\alpha_0'(\eta_\varepsilon^*)|_{L^\infty(Q)} + |\alpha'(\eta_\varepsilon^*)|_{L^\infty(Q)} + \nu^2 + M_\theta)(|z_\varepsilon|_\sV + |p_\varepsilon|_\sH + |\theta_\varepsilon^* - \theta_\ad|_\sH) \bigr).
\end{equation}
Therefore, we find a subsequence $\{ \varepsilon_n \}_{n=1}^\infty$ (not relabeled) together with a pair of functions $[p^\circ, z^\circ]$, functions $\sigma^\circ, \xi^\circ \in \sH$ and a distribution $\mathfrak{z}^\circ \in \sW^*$ such that
\begin{equation}
  \left\{ \begin{aligned}
    &p_n^* := p_{\varepsilon_n}^* \to p^\circ \mbox{ in } C([0,T];H), \mbox{ weakly in } W^{1,2}(0,T;V),
    \\
    &z_n^* := z_{\varepsilon_n}^* \to z^\circ \mbox{ weakly-$*$ in } L^\infty(0,T;V),
  \end{aligned} \right. \label{NC0_06}
\end{equation}
\begin{equation}
  \xi_n^* := \partial_t \theta_n^* z_n^* \to \xi^\circ, \, \sigma_n^* := \nabla \gamma_{\varepsilon_n}(\nabla \theta_n^*) \cdot \nabla z_n^* \to \sigma^\circ \mbox{ weakly in } \sH, \label{NC0_07}
\end{equation}
\begin{equation}
  \mathfrak{z}_n^* := -\diver(\alpha(\eta_n^*) \nabla^2 \gamma_{\varepsilon_n}(\nabla \theta_n^*) \nabla z_n^*) \to \mathfrak{z}^\circ \mbox{ weakly in } \sW^*, \mbox{ as } n \to \infty. \label{NC0_08}
\end{equation}

Now, by virtue of \eqref{NC0}--\eqref{adj_2}, $[u_{\varepsilon_n}^*, v_{\varepsilon_n}^*]$ and $[p_n^*, z_n^*]$ satisfy the following variational inequalities:
\begin{gather}
  (L_u p_n^* + M_u u_{\varepsilon_n}^*, h - u_{\varepsilon_n}^*)_\sH \geq 0, \ \mbox{ for any } h \in \sK, \label{NC0_091}
  \\
  (L_v z_n^* + M_v v_{\varepsilon_n}^*, k)_\sH \geq 0, \ \mbox{ for any } k \in \sH, \label{NC0_09}
\end{gather}
\begin{gather}
  \int_I (-\partial_t p_n^*(t) + \lambda_n^*(t) p_n^*(t) + \alpha_0'(\eta_n^*(t)) \xi_n^*(t) + \alpha'(\eta_n^*(t)) \sigma_n^*(t), \varphi)_H \,dt
  \\
  + \int_I (\nabla (p_n^* - \mu^2 \partial_t p_n^*)(t), \nabla \varphi)_{[H]^N} \,dt = \int_I (M_\eta (\eta_n^* - \eta_\ad)(t), \varphi)_H \,dt, \label{NC0_10}
  \\
  \mbox{ for any } \varphi \in V, \mbox{ and any open interval } I \subset (0,T), \mbox{ with } p_n^*(T) = 0,
\end{gather}
and
\begin{gather}
  (a_n^* z_n^*, \partial_t \psi)_\sH + \langle \mathfrak{z}_n^*, \psi \rangle_\sW + (\nu^2 \nabla z_n^*, \nabla \partial_t \psi)_{[\sH]^N} + (\omega_n^* p_n^*, \nabla \psi)_{[\sH]^N}
  \\
  = (M_\theta(\theta_n^* - \theta_\ad), \psi)_\sH, \ \mbox{ for any } \psi \in \sW. \label{NC0_11}
\end{gather}
The convergences \eqref{NC0_01} and \eqref{NC0_06} enable us to see that
\begin{align}
  0 &\leq (L_u p^\circ + M_u u^\circ, h)_\sH - (L_u p^\circ, u^\circ) - \varliminf_{n \to \infty} |\sqrt{M_u} u_{\varepsilon_n}^*|_\sH^2
  \\
  &\leq (L_u p^\circ + M_u u^\circ, h - u^\circ)_\sH, \ \mbox{ for any } h \in \sK.
\end{align}
Also, letting $n \to \infty$ in \eqref{NC0_09} yields that
\begin{equation}
  (L_v z^\circ + M_v v^\circ, k)_\sH \geq 0, \ \mbox{ for any } k \in \sH.
\end{equation}
Therefore, we obtain the assertions \eqref{mth4_01a}, \eqref{mth4_01b} and \eqref{mth4_02}. Additionally, in the light of \eqref{NC0_02}--\eqref{NC0_08}, letting $n \to \infty$ in \eqref{NC0_10} and \eqref{NC0_11} yields that \eqref{mth4_04} and the following variational identity:
\begin{gather}
  \int_I (-\partial_t p^\circ(t) + \lambda^\circ(t) p^\circ(t) + \alpha_0'(\eta^\circ(t)) \xi^\circ, \varphi)_H \,dt + \int_I (\alpha'(\eta^\circ(t)) \sigma^\circ(t), \varphi)_H \,dt\label{NC0_12}
  \\
  + \int_I (\nabla (p^\circ - \mu^2 \partial_t p^\circ)(t), \nabla \varphi)_{[H]^N} \,dt= \int_I (M_\eta (\eta^\circ - \eta_\ad)(t), \varphi)_H \,dt,
  \\
  \mbox{ for any } \varphi \in V, \mbox{ and any open interval } I \subset (0,T),
\end{gather}
with
\begin{equation}
  p^\circ(T) = \lim_{n \to \infty} p_n^*(T) = 0 \mbox{ in } H.
\end{equation}
Since $I$ is arbitrary, we obtain \eqref{mth4_03}.

Thus, we complete the proof of Main Theorem \ref{mth4}. \noeqref{NC0_03,NC0_04,NC0_05,NC0_06,NC0_07,NC0_08,NC0_091}
\qed


\end{document}